%% file: main.tex
\documentclass[submission,copyright,creativecommons]{eptcs}
\pdfoutput=1
\providecommand{\event}{ACT 2022}

\input{preamble.tex}

\usepackage{xspace}
\def\yo{\ensuremath{\mathbb Y}\xspace}

\title{On the Pre- and Promonoidal Structure of Spacetime}
\author{James Hefford
\institute{University of Oxford, UK}
\email{james.hefford@cs.ox.ac.uk}
\and
Aleks Kissinger
\institute{University of Oxford, UK}
\email{aleks.kissinger@cs.ox.ac.uk}
}

\def\titlerunning{On the Pre- and Promonoidal Structure of Spacetime}
\def\authorrunning{J. Hefford \& A. Kissinger}

\begin{document}

\maketitle

\begin{abstract}
The notion of a joint system, as captured by the monoidal (a.k.a.\ tensor) product, is fundamental to the compositional, process-theoretic approach to physical theories. Promonoidal categories generalise monoidal categories by replacing the functors normally used to form joint systems with profunctors. Intuitively, this allows the formation of joint systems which may not always give a system again, but instead a generalised system given by a presheaf. This extra freedom gives a new, richer notion of joint systems that can be applied to categorical formulations of spacetime. Whereas previous formulations have relied on partial monoidal structure that is only defined on pairs of independent (i.e.\ spacelike separated) systems, here we give a concrete formulation of spacetime where the notion of a joint system is defined for any pair of systems as a presheaf. The representable presheaves correspond precisely to those actual systems that arise from combining spacelike systems, whereas more general presheaves correspond to virtual systems which inherit some of the logical/compositional properties of their ``actual'' counterparts. We show that there are two ways of doing this, corresponding roughly to relativistic versions of conjunction and disjunction. The former endows the category of spacetime slices in a Lorentzian manifold with a promonoidal structure, whereas the latter augments this structure with an (even more) generalised way to combine systems that fails the interchange law.
\end{abstract}

\section{Introduction}

Categorical approaches to the modelling of structures of spacetime have become increasingly rich topics of study leading to both the development of new mathematics and a greater understanding of the underlying structures of our theories of physics.
Nevertheless, the precise categorical structures that should be present in a model of spacetime are far from settled.
Monoidal structure is a common requirement, being a key part of Categorical Quantum Mechanics \cite{coecke_cqm} and of many approaches to Topological Quantum Field Theory \cite{baez_tqft}.
The key physical argument for the assumption of monoidal structure is simple: if one has a pair of systems, then one should be able to put them together and consider the composite as a new system.

While this assumption may be ideal in abstract process theories, say where one wishes to model arbitrary qubits as in the ZX-calculus \cite{coecke_zx}, when we turn our attention to \textit{decompositional} approaches to modelling physical systems \cite{coecke_causalcats}, it becomes apparent that the universe does not behave in a fully monoidal fashion.
Rather than starting with a collection of existent systems and presupposing that it is possible to join them together arbitrarily, we start with a global system - the whole of spacetime - and carve out systems with the hope of recovering some fragment of compositional structure.

In such a framework, the tensor becomes problematic, for instance, if we pick a particular system, say a specified qubit $A$, it is clearly not possible to form the product $A\otimes A$ in the usual sense, for what would it mean to consider the composite of a system with itself?
Indeed, the fundamental issue here is trying to tensor two objects that are not independent and that can influence each other in non-trivial ways; we would also have issues taking the tensor of timelike separated systems, or of mixed systems whose environments are not causally disjoint.

There are two main obstructions to hoping for a total tensor product on a category modelling spacetime regions.
Firstly, one would like the objects of the category to have a physical interpretation as systems existing in reality.
It can often be the case though that no such physical system exists for the composite of physically reasonable systems.
For instance, if we take the objects of our category to represent slices of spacetime - closed spacelike subsets of a Lorentzian manifold - when we try to join two slices together they will not form another slice unless the original slices were causally separated.

Secondly, functoriality can fail and one often finds that the interchange law does not hold:
\begin{equation}\label{eq:interchange}
  (g\otimes 1)(1\otimes f) \neq (1\otimes f)(g\otimes 1)
\end{equation}
while functoriality in each side of the tensor still holds $(1\otimes f')(1\otimes f) = (1\otimes f'f)$ and $(g'\otimes 1)(g\otimes 1) = (g'g\otimes 1)$.
This occurs because the systems involved in the tensor may not be independent - they might causally influence each other or possess a shared environment.
Thus the casual ordering of $f$ and $g$ is vitally important.

One possible route forwards could be to define the tensor only partially.
It was noted in \cite{coecke_causalcats} that one can recover a partial monoidal structure where the tensor product is only defined on regions of spacetime that are causally separated.
A group theoretic approach was taken in \cite{gogioso_church} where the resulting category has partial monoidal structure defined only on compatible systems, which requires both the causal separation of systems and also their coupled environments.
Another approach starting with a poset modelling the causal relationships of spacetime events \cite{gogioso_functorial}, resulted in partial monoidality, again only defined on causally separated systems.
Partial monoidality, due to similar causality obstructions has appeared in a proposal for modelling the Wolfram model \cite{gorard}.

Outside of applications to physics, partially monoidal categories have made an appearance in \cite{balco_thesis,balco_partiallymonoidal,balco_nominal} where it was noted that the category of finite subsets of some given set $N$ has a partially monoidal structure given by the union of disjoint sets.
The authors develop a string diagrammatic language dubbed \textit{nominal string diagrams}, where wires are labelled with elements from the fixed set $N$.
There are similarities between this and the present work - our decompositional approach to physics also has a fixed global set from which we label all systems (a manifold $\mathcal{M}$) and the partial monoidal structure of spacetime slices developed here is also given by unions and intersections of sets.
On the other hand, there is a major point of difference between our approach and that of Balco {et. al.}.
While they made the partial monoidal structure total by working with categories internal to a monoidal category, we aim to totalise the partial monoidal structure by working with the presheaves of our category.

We propose the usage of weakenings of monoidal categories in the form of \textit{pro}monoidal \cite{day} and \textit{pre}monoidal \cite{power_premonoidal} categories to model causal curves in spacetime.
Premonoidal categories are like monoidal ones but dropping the interchange law \eqref{eq:interchange}.
They were developed for modelling computational semantics with side-effects and have been used previously to model spacetime particularly in relation to Algebraic Quantum Field Theories \cite{comeau,blute}, where it was argued one could use them to model the Einstein causality condition.
Here, we reinforce their point and argue that the lack of bifunctoriality seems to be fundamental in a decompositional approach to spacetime.
Premonoidal categories have also appeared elsewhere in applications to petri nets \cite{baez_petri}.

Promonoidal categories are loosely like monoidal categories into the presheaf category.
To our knowledge they have not been directly used in a model of spacetime before.
Here, we use them to extend the partial monoidality of spacetime to a total tensor by allowing us to assign useful mathematical objects to otherwise physically problematic ones.
For instance, the union of two slices of spacetime is another region of the manifold but not necessarily a slice, thus lacking physical interpretation.
We can assign the union a presheaf, with these presheaves being representable whenever the union is another slice.
The non-representable presheaves can be thought to act like ``virtual systems,'' they carry useful information but are not physically meaningful.

In sections \ref{sec:promonoidal} and \ref{sec:premonoidal} we recap promonoidal and premonoidal categories respectively.
In section \ref{sec:slices} we introduce toy categories $\Slice$ and $\Space$ of causal curves in spacetime before showing in section \ref{sec:slice_promonoidal} that $\Slice$ is a promonoidal category under the operation of taking intersections of sets of causal curves.
In section \ref{sec:slice_union} we discuss the operation of taking unions of sets of causal curves and demonstrate that this gives a premonoidal structure on $\Space$ while $\Slice$ combines the structures of promonoidal and premonoidal categories.
Under either of the tensor-like structures on $\Slice$ we prove that the presheaves assigned to the tensors are representable if and only if the slices are jointly spacelike and in doing so show that we recover a type of partial tensor product on causally separated regions.
In the final section \ref{sec:logic} we give the beginnings of a physical interpretation to the operations on $\Slice$ as capturing a kind of logical conjunction and disjunction of predicates about particles in spacetime.

\section{Preliminaries}
\subsection{Promonoidal Categories}\label{sec:promonoidal}
Before we introduce the formal definition of a promonoidal category let us comment on the intuition we hope to capture.

In a monoidal category $\cat{C}$, the tensor product of two objects of $\cat{C}$ returns another object in $\cat{C}$, that is, it is a functor $\cat{C}\times\cat{C}\morph{}\cat{C}$.
Returning to the example of a category of spacetime slices, it is problematic to assign an object of $\cat{C}$ to the tensor product whenever the regions of spacetime are timelike separated.
The best we could hope for would be a \textit{partial} monoidal structure which is only defined when regions are spacelike separated.
Perhaps it might be possible though to assign the tensor of timelike separated regions to be a different sort of object, one that lives outside the category $\cat{C}$?
What is a sensible choice of such ``external'' objects and how can we ensure that they work together compatibly such that we might describe the overall structure as something like a tensor product?

We will investigate the usage of promonoidal categories to deal with the aforementioned issues.
Rather than assign an object of $\cat{C}$ to the tensor product, we assign it a \textit{presheaf}: a functor $\opcat{\cat{C}}\morph{}\set$.
Presheaves are nicely-behaved mathematical objects: they form a category $[\opcat{\cat{C}},\set]$ where the morphisms are natural transformations between the presheaves, and the Yoneda lemma provides a way of embedding of $\cat{C}$ fully and faithfully into its presheaves $\yo:\cat{C}\morph{}[\opcat{\cat{C}},\set]$.
The image of this functor consists of the \textit{representable} presheaves which are of the form $\yo_A\cong\cat{C}(-,A)$ for some object $A$ of $\cat{C}$.

By working with promonoidal categories we are able to assign the tensor a presheaf $(A\otimes B)(-):\opcat{\cat{C}}\morph{}\set$, and in doing so, work with otherwise undefinable tensor products.
Since $\cat{C}$ embeds into its presheaves, we do not lose any ability to still assign some tensor products to essentially be objects of $\cat{C}$.
Indeed, when the tensor product of objects of $\cat{C}$ is a representable presheaf, $(A\otimes B)(-) \cong \cat{C}(-,C)$ we can identify $A\otimes B$ with $C$ under the Yoneda embedding.
In this way, promonoidal categories are like partially monoidal ones - when the presheaf is representable we essentially have an object of $\cat{C}$ again - but rather than the tensor being undefined elsewhere we can still assign otherwise ``untensorable'' objects a non-representable presheaf.
For a more detailed discussion of the connections between partially monoidal and promonoidal categories see appendix \ref{sec:partiallymonoidal}.

Now, let us start with some core definitions concerning profunctors and their composition.
A more comprehensive study can be found in e.g. \cite{loregian_coend}.

\begin{defn}[Profunctor]
  A profunctor $P:\cat{C}\pmorph\cat{D}$ is a functor $\opcat{\cat{D}}\times\cat{C}\morph{}\set$.
\end{defn}

Profunctors generalise functors in a similar way to how relations generalise functions between sets - profunctors are like ``relations between categories,''
(note that a relation $A\sim B$ is equivalently a function out of the cartesian product of the sets $A\times B\morph{} \{0,1\}$).
We will often use a shorthand Einstein-style notation for profunctors writing $P(d,c) = P^d_c$, with subscripts for covariant variables and superscripts for contravariant ones.

\begin{defn}[Cowedge, Coend]
  Given a profunctor $P:\opcat{\cat{C}}\times\cat{C}\morph{}\set$, a cowedge $(d,w)$ for $P$ is an object $d$ of $\set$ together with arrows $w_c:P(c,c)\morph{}d$ making the following diagram commute for all $f$:
  \begin{equation*}
    \begin{tikzcd}
      d & P(c,c) \ar[l,"w_c"']\\
      P(c',c') \ar[u,"w_{c'}"] & P(c',c) \ar[l,"P{(c',f)}"] \ar[u,"P{(f,c)}"']
    \end{tikzcd}
  \end{equation*}
  The coend of $P$ is a universal cowedge $(\int^c P(c,c),\text{copr})$: this is the cowedge such that all other cowedges factorise uniquely through it:
  \begin{equation*}
    \begin{tikzcd}
      d \\
      & \int^c P(c,c)  \ar[ul, dashed]  & P(c,c) \ar[l,"\text{copr}_c"'] \ar[ull, "w_{c}", bend right, swap]\\
      & P(c',c') \ar[u,"\text{copr}_{c'}"] \ar[uul, "w_{c'}", bend left] & P(c',c) \ar[l,"P{(c',f)}"] \ar[u,"P{(f,c)}"']
    \end{tikzcd}
  \end{equation*}
\end{defn}

Coends have a series of nice properties which help to justify the use of an integral symbol to represent them.
Firstly, they satisfy a Fubini-style law allowing us to commute coends:
\begin{equation*}
  \int^c\int^d P(c,c,d,d) \cong \int^{(c,d)\in\cat{C}\times\cat{D}} P(c,c,d,d) \cong \int^d\int^c P(c,c,d,d)
\end{equation*}
Secondly, the Yoneda lemma implies the following identities, sometimes known as the ninja Yoneda lemma:
\begin{equation}\label{eq:ninja}
  \int^c \cat{C}(-,c)\times F(c) \cong F(-) \hspace{2cm} \int^c G(c)\times\cat{C}(c,-) \cong G(-)
\end{equation}
for any functors $F:\opcat{\cat{C}}\morph{}\set$ and $G:\cat{C}\morph{}\set$.
So the hom-profunctor behaves ``like a Dirac-delta function''.

\begin{defn}[Composition of Profunctors]
  Given profunctors $P:\cat{C}\pmorph\cat{D}$ and $Q:\cat{D}\pmorph\cat{E}$, their composite is given by taking the coend
  \begin{equation*}
    (Q\circ P)(-,=) = \int^d Q(-,d)\times P(d,=) : \cat{C} \pmorph \cat{E}
  \end{equation*}
\end{defn}

This coend can be characterised as the coequaliser:
\begin{equation*}
  \bigsqcup_{f:d\morph{}d'} Q(-,d)\times P(d',=) \ \text{\Large$\rightrightarrows$} \ \bigsqcup_{d} Q(-,d)\times P(d,=) \longrightarrow \int^d Q(-,d)\times P(d,=)
\end{equation*}
where the coequalised pair ``act by $f$ on the left and right under the profunctor''.
We can think of the resulting quotient set $(Q\circ P)(e,c)$ as equivalence classes of pairs $(q,p)$ where $q\in Q(e,d)$ and $p\in P(d,c)$ under the relations $(Q(e,f)(q),p) \sim (q,P(f,c)(p))$.
We will refer to these as the ``sliding'' relations since it is as though we can slide $f$ from one side to the other (up to changing $Q$ and $P$).

The composition of profunctors will be written as $(Q\circ P)(e,c) = Q^e_d P^d_c$ in the Einstein notation, where instead of the summation convention we have a ``coend convention'' - repeated indices, once covariant and once contravariant, are to be coend-ed out.
In this way, one also sees the similarity between profunctor composition and matrix multiplication.

Categories, profunctors and natural transformations form a monoidal bicategory $\Prof$ where the monoidal product acts as $\cat{C}\times\cat{D}$ on 0-cells and as $(P\times Q)(c,d,e,f) = P(c,d)\times Q(e,f)$ on 1-cells.
The hom-profunctors play a special role in $\Prof$: they are the identity 1-cells by the ninja Yoneda lemma \eqref{eq:ninja}.

We are now in a position to define promonoidal categories.

\begin{defn}[Promonoidal Category \cite{day,day_thesis}]
  A category $\cat{C}$ is promonoidal if it is equipped with
  \begin{itemize}
    \item a tensor product profunctor $\otimes:\cat{C}\times\cat{C}\pmorph\cat{C}$
    \item a unit profunctor $I:1\pmorph\cat{C}$, i.e. a presheaf $I:\opcat{\cat{C}}\morph{}\set$
  \end{itemize}
  together with natural isomorphisms $\otimes(\otimes\times 1) \overset{\alpha}{\cong} \otimes(1\times \otimes)$ and $\otimes(\otimes\times I) \overset{\rho}{\cong} 1 \overset{\lambda}{\cong} \otimes(I\times \otimes)$ subject to the triangle and pentagon coherence conditions similar to a monoidal category.
  A promonoidal category is \textit{strict} when the coherence isomorphisms are identities.
  A promonoidal category is \textit{symmetric} when there is a natural isomorphism $\sigma_{ABC}:\otimes^A_{BC}\morph{}\otimes^A_{CB}$ satisfying the hexagon equation.
\end{defn}

\begin{remark}
  A very concise definition of a promonoidal category $\cat{C}$ is as a pseudomonoid in $\mathsf{Prof}$.
\end{remark}

There are many similarities between the definitions of promonoidal and monoidal categories.
One can think of promonoidal categories as what we get when we ``upgrade'' the functors of a monoidal category to profunctors.
This really is an upgrade since every functor induces two profunctors by taking its covariant or contravariant Yoneda embeddings.
Furthermore, by the following result we can consider promonoidal categories as strictly more general than monoidal ones.
\begin{thm}[\cite{day,day_thesis}]\label{prop:monoidal_promonoidal}
  All monoidal categories $(\cat{C},\boxtimes,J)$ are promonoidal categories where we define the tensor profunctor as $(A\otimes B)(-) := \cat{C}(-,A\boxtimes B)$ and the unit profunctor as $I(-):= \cat{C}(-,J)$.
  Conversely, a promonoidal category whose tensor and unit are everywhere representable is a monoidal category.
\end{thm}

We will mostly think of the tensor product profunctor $\otimes:\opcat{\cat{C}}\times\cat{C}\times\cat{C}\morph{}\set$ in its curried form as a functor into presheaves, $\otimes:\cat{C}\times\cat{C}\morph{}[\opcat{\cat{C}},\set]$ and in an abuse of notation we freely switch between using $\otimes$ for the tensor product in its three different forms (as a profunctor, a functor into $\set$ and a functor into presheaves) so long as it is clear which we mean.

\subsection{Premonoidal Categories}\label{sec:premonoidal}
Alongside promonoidal categories, the other monoidal-like structures in this article are premonoidal categories.
Premonoidal categories are a weakening of monoidal categories to allow for situations when one can join objects together but each half of the tensor is only individually functorial, that is, while it is the case that $(g'\otimes 1)(g\otimes 1) = (g'g\otimes 1)$ and $(1\otimes f')(1\otimes f) = (1\otimes f'f)$ we have the following inequality:
\begin{equation*}
  \input{./figures/premonoidal.tikz}
\end{equation*}
These categories were originally introduced to model computational semantics with side-effects \cite{power_premonoidal} but we expect categories of causal curves to have similar structure.
If $f$ and $g$ act on slices which are timelike separated or have a non-trivial intersection, then their causal ordering can be vitally important; $f$ could change the state space in ways that later influence $g$ or vice-versa.
These ``hidden'' influences between maps can be seen to be somewhat akin to the side-effects in the computational semantics for which premonoidal categories were originally intended.

A premonoidal category has for each object $X$, a pair of functors $X\rtimes -$ and $-\ltimes X$ acting as the left and right parts of the tensor product, together with compatibility between their actions on objects.
More precisely:
\begin{itemize}
  \item for each pair of objects $X$ and $Y$ of $\cat{C}$ there is an assigned object $X\boxtimes Y$ of $\cat{C}$,
  \item for each object $X$ of $\cat{C}$, there is a functor $X\rtimes - :\cat{C}\morph{}\cat{C}$ acting on objects as $X\rtimes Y = X\boxtimes Y$,
  \item for each object $Y$ of $\cat{C}$, there is a functor $-\ltimes Y:\cat{C}\morph{}\cat{C}$ acting on objects as $X\ltimes Y = X\boxtimes Y$.
\end{itemize}

There is no compatibility condition between the left and right parts on morphisms, so in general it will be the case that $(f\ltimes Y')(X\rtimes g) \neq (Y\rtimes g)(f\ltimes X')$ for $f:X\morph{}Y$, $g:X'\morph{}Y'$.
Pairs of morphism for which such equalities hold, we can think of as acting like the normal tensor and can safely denote $f\otimes g$.
In particular, there is a special name for those morphisms which commute with all others:

\begin{defn}[Central Morphism \cite{power_premonoidal}]
  A morphism $f:X\morph{}Y$ is central if and only if for all $g:X'\morph{}Y'$, the following two diagrams commute:
  \begin{equation*}
    \begin{tikzcd}
      X\otimes X' \ar[r,"X\rtimes g"] \ar[d, "f\ltimes X'"'] & X\otimes Y' \ar[d,"f\ltimes Y'"] \\
      Y\otimes X' \ar[r,"Y\rtimes g"'] & Y\otimes Y'
    \end{tikzcd}
    \hspace{1cm}
    \begin{tikzcd}
      X'\otimes X \ar[r,"X'\rtimes f"] \ar[d, "g\ltimes X"'] & X'\otimes Y \ar[d,"g\ltimes Y"] \\
      Y'\otimes X \ar[r,"Y'\rtimes f"'] & Y'\otimes Y
    \end{tikzcd}
  \end{equation*}
\end{defn}

In addition to the above data, a premonoidal category needs associativity and unit natural isomorphisms which are central:

\begin{defn}[Premonoidal Category \cite{power_premonoidal}]
  A category $\cat{C}$ is premonoidal if it is equipped with left and right tensor functors $X\rtimes -$ and $-\ltimes Y$ for each $X$ and $Y$, such that they are compatible on objects, together with:
  \begin{itemize}
    \item a unit object $I$ with central isomorphisms $\lambda_X:X\otimes I\morph{}X$ and $\rho_X:I\otimes X\morph{}X$ for each $X$,
    \item a central isomorphism $\alpha_{XYZ}:(X\otimes Y)\otimes Z\morph{} X\otimes (Y\otimes Z)$ for each triple $X,Y$ and $Z$,
  \end{itemize}
  such that the triangle and pentagon equations hold and so that the naturality squares for $\alpha,\lambda$ and $\rho$ hold.
  A premonoidal category is \textit{strict} when the coherence isomorphisms are identities.
\end{defn}

It is possible to combine the left and right tensor functors $X\rtimes -$ and $-\ltimes Y$ into a single functor $\cat{C} \funny \cat{C}\morph{}\cat{C}$ from the \textit{funny tensor product} \cite{foltz}.
A concise definition of the funny tensor is as follows,

\begin{defn}[Funny tensor product \cite{weber}]
  The funny tensor product $\cat{C}\funny\cat{D}$ is given by the following pushout
  \begin{equation}\label{funnytensor}
    \begin{tikzcd}
      \cat{C}_0\times\cat{D}_0 \ar[r,"1\times i_\cat{D}",hookrightarrow] \ar[d, "i_\cat{C}\times 1"',hookrightarrow] & \cat{C}_0\times\cat{D} \ar[d] \\
      \cat{C}\times\cat{D}_0 \ar[r] & \cat{C}\funny\cat{D} \arrow[ul, phantom, "\ulcorner", very near start, xshift= 0.2cm, yshift= 0cm]
    \end{tikzcd}
  \end{equation}
  where $\cat{C}_0$ and $\cat{D}_0$ are the discrete categories of the objects of $\cat{C}$ and $\cat{D}$ respectively.
\end{defn}

Explicitly, the category $\cat{C}\funny\cat{D}$ has as objects pairs $(c,d)$ of an object $c$ of $\cat{C}$ and $d$ of $\cat{D}$.
The morphisms are generated by freely composing $(f;1):(c,d)\morph{}(c',d)$ where $f:c\morph{}c'$ in $\cat{C}$ and $(1;g):(c,d)\morph{}(c,d')$ where $g:d\morph{}d'$ in $\cat{D}$ with the rule that compositions exclusively in $\cat{C}$ or $\cat{D}$ may be contracted: $(f';1)(f;1)=(f'f;1)$ and $(1;g')(1;g) = (1;g'g)$ but $(f;1)(1;g)\neq(1;g)(f;1)$ and thus there is no sensible notion of ``$(f;g)$''.
There is a oplax monoidal functor $\cat{C}\funny\cat{D}\morph{}\cat{C}\times\cat{D}$ induced by the universal property of the pushout, which forces the interchange squares to commute.

\section{A Category of Spacetime Slices}\label{sec:slices}
The aim of the remainder of this article is to develop a toy category of spacetime slices and causal curves and then demonstrate that it exhibits both premonoidal and promonoidal structures.

\subsection{Spacetimes and Causal Curves}
From now on we fix a connected Lorentzian manifold $\mathcal{M}$ with metric $g$.
A tangent vector $X$ is said to be \textit{spacelike}, \textit{timelike} or \textit{null} if $g(X,X)>0, g(X,X)<0$ or $g(X,X)=0$, respectively.
$\mathcal{M}$ is said to be \textit{time-orientable} if it has a non-vanishing timelike vector field and the timelike tangent vectors at each point can be divided (in a continuous fashion) into two classes: a \textit{future-directed} and a \textit{past-directed} class.
We assume that $\mathcal{M}$ is time-orientable and fix a time-orientation.
The assumptions we make of our spacetime are fairly weak causality-wise, and are weaker than those of past- and future-distinguishability \cite{malament,kronheimer} (which was assumed by \cite{gogioso_functorial}) and certainly weaker than the existence of a Cauchy slice (equivalently global hyperbolicity) \cite{geroch}.
As a result we have not ruled out the existence of closed timelike curves in the spacetime.

A simple example of the kinds of manifolds we are interested in is Minkowski space $\reals^{n+1}$ equipped the metric $g(X,X)=|x|^2-t^2$ for $X=(t,x)$.
The timelike vectors are those $(t,x)$ where $t^2>|x|^2$, of which there are two classes $t>|x|$ and $t<-|x|$ consisting of vectors which point forwards and backwards in time, respectively; a timelike vector $(t,x)$ is future-directed when $t>0$ and past-directed when $t<0$.
There is no issue with restricting oneself to Minkowski space for the remainder of the article, but we note that the results hold in the fully general case.

A \textit{path} in $\mathcal{M}$ is a continuous map $\mu:\iota\morph{}\mathcal{M}$ where $\iota \subseteq \reals$ is a (possibly unbounded) real interval.
Such a path is \textit{smooth} if it is infinitely differentiable and \textit{regular} if its first derivative is non-vanishing.
A smooth regular path is \textit{causal} when the tangent vector is timelike or null at all points in the path and a causal path is \textit{future-directed} when the tangent at every point is future-directed.
For a point $x \in \mathcal M$, the set of all points $y \in \mathcal M$ with a future-directed path $x$ to $y$ is called the \textit{future light cone} of $x$, whereas the set of all points with a future-directed path from $y$ to $x$ is called the \textit{past light cone} of $x$.

Often it is more convenient to work with equivalence classes of paths, up to reparametrisation, i.e. $\mu \sim \mu'$ if and only if there exists a monotone map $r : \iota \to \iota'$ such that $\mu' \circ r = \mu$. An equivalence class of causal paths is called a \textit{causal curve}. Since being future-directed is preserved by $\sim$, we can also say a causal curve is future-directed without ambiguity.

A point $x\in\mathcal{M}$ \textit{causally precedes} another point $y\in\mathcal{M}$, written $x\prec y$, if there exists a future-directed causal curve from $x$ to $y$, or if $x=y$.
The assumption of time-orientability of $\mathcal{M}$ is not enough to ensure that $\prec$ gives a total order on points in a causal curve - for instance there could be closed timelike curves in $\mathcal{M}$ containing points $x\neq y$, for which $x\prec y$ and $y\prec x$.

A \textit{region} is any arbitrary subset $A\subseteq \mathcal{M}$ of the manifold.
Regions are too general to be useful for many practical applications, they might contain points which causally precede each other or they might have insufficient topological properties to make them well-behaved.
As a result we will be more interested in a restricted class of regions, the \textit{spacelike} regions, where for all $x,y\in \Sigma$, $x\neq y$, $x$ does not causally precede $y$ and thus there are no future-directed causal curves connecting $x$ with $y$, or $y$ with $x$.
For instance, in Minkowski space the surfaces given by fixed times $t=\tau$ are examples of spacelike sets.

\begin{defn}[Spacelike Slice]
  A spacelike slice (or simply a ``slice'') is a closed spacelike set.
\end{defn}

It is worth noting that slices may still be too weak for many applications, and it may be necessary to demand further properties of them, by working with the Cauchy slices for instance.
Whilst we do not make these restrictions in this work, in principle, there is no obstacle to applying many of the same methods to categories of more restrictive classes of slices.

We will be very interested in the causal relationship between slices $X$ and $Y$, which motivates the following definition.
\begin{defn}[Jointly Spacelike Slices]
  Slices $X$ and $Y$ are jointly spacelike if their union $X\cup Y$ is spacelike.
\end{defn}

Given regions $A,B \subseteq \mathcal{M}$, $A\neq B$, we say that a future-directed causal curve $\gamma$ with representative path $\mu:\iota\morph{}\mathcal{M}$, \textit{passes through $A$ and then $B$} if there exists a $q\in \iota$ with $\mu(q)\in B$ and for all such $q$ there exists $p\leq q\in\iota$ such that $\mu(p)\in A$.
We write $\cat{C}[A,B]$ for the set of future-directed causal curves passing through $A$ and then $B$.
We write $\cat{C}[A]:=\cat{C}[A,A]$ for the set of future-directed causal curves which pass through $A$ (with no constraint on other regions through which they must pass).
It is worth noting that a closed timelike curve $\gamma$ containing both the points $a\in A$ and $b\in B$ will be in the sets $\cat{C}[A,B]$ and $\cat{C}[B,A]$.

\subsection{A Category of Causal Curves}
With these definitions in place we can define the following categories of slices and regions of spacetime:

\begin{defn}[$\Slice, \Space$]
  The category $\Slice$ has as objects slices $X\subset \mathcal{M}$ (closed spacelike sets).
  For two slices $X,Y \subset \mathcal{M}$, the homset $\Slice(X,Y):= \mathcal{P}(\cat{C}[X,Y])$ is the powerset of $\cat{C}[X,Y]$, that is, a morphism $X\morph{}Y$ is a set of future-directed causal curves through $X$ then $Y$.
  Given two subsets $S:X\morph{}Y$ and $T:Y\morph{}Z$, their composition is given by intersection: $T\circ S:= T\cap S \subset \cat{C}[X,Z]$.
  The identity morphism $1_X:X\morph{}X$ is given by the set $\cat{C}[X,X]$ of all curves through $X$.

  The category $\Space$ has as objects arbitrary regions $A\subseteq\mathcal{M}$.
  All other data is as $\Slice$.
\end{defn}

\begin{figure}[h]
  \centering
  \input{./figures/morphism1.tikz}\qquad\qquad
  \input{./figures/compose.tikz}
  \caption{\label{fig:morphism} \textit{Left:} A morphism in the category $\Slice$ is a set of causal curves passing first through $X$ then through $Y$. \textit{Right:} Composition of two morphisms in $\Slice$ via intersection. Note that in both pictures, past as future light cones of slices are depicted as dotted lines, and sets of many causal curves are depicted as filled-in regions.}
\end{figure}
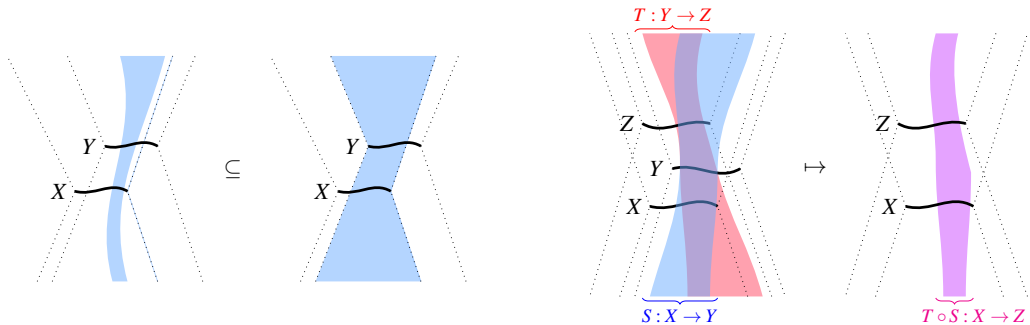

\begin{prop}
  $\Slice$ and $\Space$ are categories.
\end{prop}
\begin{proof}
  Composition is associative because intersection is.
  Given a set of causal curves $S:X\morph{}Y$, by definition all curves in $S$ pass through $X$, thus we see $S\circ 1_X = S\cap \cat{C}[X,X] = S$.
  Similarly for the left composition with identity morphisms.
\end{proof}

Now we examine a few basic categorical properties of $\Slice$ and $\Space$.
\begin{prop}
  $\Slice$ and $\Space$ have equalisers and coequalisers, given by the complement of the symmetric difference.
\end{prop}
\begin{proof}
  Take $f,g:A\morph{}B$.
  This pair of parallel arrows is equalised by $(f\mathbin{\triangle} g)^c:A\morph{}A$ and coequalised by $(f\mathbin{\triangle} g)^c:B\morph{}B$ where $(f\mathbin{\triangle} g)^c = \cat{C}[A,B]\backslash (f\mathbin{\triangle} g) = (f\cup g)^c \cup (f\cap g)$.
  Any other arrow $h$ making the parallel pair $f$ and $g$ equal factorises uniquely via $(f\mathbin{\triangle} g)^c$ because this morphism contains every causal curve that is in both $f$ and $g$, or neither.
  Thus $h$ must be a subset of $(f\mathbin{\triangle} g)^c$.
\end{proof}
It is interesting that equalisers and coequalisers essentially coincide in $\Slice$ - in part this is down to the fact that composition is, up to types, commutative - e.g.\ for endomorphisms $f\circ g = g\circ f$.

\begin{prop}\label{prop:sliceprocopro}
  Let $X$ and $Y$ be jointly spacelike slices with $X\cap Y=\varnothing$.
  Then the product and coproduct of $X$ and $Y$ exist in $\Slice$ and are given by the set theoretic union $X\times Y = X\oplus Y = X\cup Y$.
\end{prop}
\begin{proof}
  Proof given in appendix \ref{proof:sliceprocopro}.
\end{proof}

While we do have products and coproducts of non-intersecting jointly spacelike slices in $\Slice$, the (co)products of other regions e.g.\ timelike separated regions and of intersecting slices do not exist.
These regions are the main issue preventing the set theoretic union from being a monoidal structure on $\Slice$.

\begin{prop}
  $\Slice$ is \textbf{not} a monoidal category under a monoidal product given by taking the union of regions and curves $X\otimes Y := X\cup Y$ and $S\otimes T := S\cup T$.
\end{prop}
\begin{proof}
  The union of slices is not always a slice so $X\cup Y$ may not be an object of $\Slice$.
  For the occasions when it is, $\otimes$ cannot in general be bifunctorial.
  For arbitrary $S:X\morph{}Y$, $S':Y\morph{}Z$, $T:X'\morph{}Y'$ and $T':Y'\morph{}Z'$, we have $(S'\otimes T')\circ (S\otimes T) = (S'\cup T')\cap(S\cup T) \supset (S'\cap S)\cup(T'\cap T) = (S'\circ S)\otimes(T'\circ T)$.
\end{proof}

One might hope that by relaxing the sorts of objects we are considering and working instead with the category $\Space$, we could find a monoidal product given by union.
Whilst this resolves the issue of the non-existence of the object $X\cup Y$ for arbitrary $X$ and $Y$, we still find that the union cannot be bifunctorial and thus $\Space$ is also not a monoidal category under union.

We also cannot hope that $\Slice$ or $\Space$ are monoidal categories under intersection because there exist causally connected slices which have an empty intersection:
\begin{prop}
  $\Slice$ and $\Space$ are \textbf{not} monoidal categories under a monoidal product given by taking the intersection of regions and curves $X\otimes Y := X\cap Y$ and $S\otimes T := S\cap T$.
\end{prop}
\begin{proof}
  Suppose $X$ and $Y$ are causally connected slices so $\cat{C}[X,Y]\neq\varnothing$ but with $X\cap Y=\varnothing$.
  Then $1_X\otimes 1_Y = \cat{C}[X,X]\cap\cat{C}[Y,Y] \neq \varnothing$ because there exists a causal curve passing through $X$ and $Y$.
  On the other hand we see that $1_{X\cap Y} = 1_\varnothing = \varnothing$.
\end{proof}

In the following sections we will show that while $\Slice$ and $\Space$ are not monoidal categories in either of these ways, $\Slice$ is a promonoidal category under intersection.
Under union, $\Space$ is premonoidal while $\Slice$ combines both promonoidal and premonoidal structures.

\section{A Promonoidal Structure on \sf{Slice}}\label{sec:slice_promonoidal}
We now aim to show that $\Slice$ is a promonoidal category under intersection, that is, it is equipped with a tensor product functor $\Slice\times\Slice\morph{}[\opcat{\Slice},\set]$ and unit presheaf $\opcat{\Slice}\morph{}\set$ subject to associativity and unit laws.

To each pair of objects $X$ and $Y$ we assign the presheaf $(X\oand Y)(-):\opcat{\cat{\Slice}}\morph{}\set$ which sends a slice $Z$ to the powerset of causal curves which pass through $Z$ and then both $X$ and $Y$
\begin{equation*}
  (X\oand Y)(Z) := \mathcal{P}(\cat{C}[Z,X]\cap\cat{C}[Z,Y])
\end{equation*}
On morphisms $S:Z'\morph{}Z$ this presheaf acts by intersection:
\begin{equation*}
  (X\oand Y)(S) :(X\oand Y)(Z) \morph{} (X\oand Y)(Z') :: C\mapsto C\cap S
\end{equation*}

\begin{lem}
  $(X\oand Y)(-)$ is a presheaf.
\end{lem}

\begin{proof}
  $(X\oand Y)(1_Z)::C\mapsto C\cap 1_Z = C$ because every curve in $(X\oand Y)(Z)$ passes through $Z$.
  Thus $(X\oand Y)(1_Z) = 1_{(X\oand Y)(Z)}$.
  Now $(X\oand Y)(T) \circ (X\oand Y)(S) :: C\mapsto C\cap S \mapsto (C\cap S)\cap T$ while $(X\oand Y)(S\circ T) :: C\mapsto C\cap(S\cap T)$ and these are equal by the associativity of intersection.
\end{proof}

\begin{figure}
  \centering
  \input{./figures/conjunction.tikz}
  \qquad\qquad
  \input{./figures/disjunction.tikz}
  \caption{\label{fig:conj} \textit{Left:} An element $S \in (X \oand Y)(Z)$, as defined in Section~\ref{sec:slice_promonoidal}. \textit{Right:} An element $T \in (X \oor Y)(Z)$, as defined in Section~\ref{sec:slice_union}.}
\end{figure}
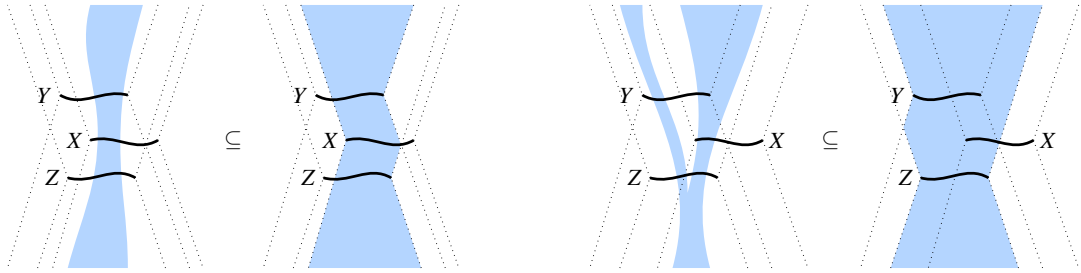

To each $(S,T):(X,Y)\morph{}(X',Y')$ we are required to assign a natural transformation between the presheaves $S\oand T: (X\oand Y)(-)\implies(X'\oand Y')(-)$.

For $S:X\morph{}X'$ there is a natural transformation with components
\begin{equation*}
  (S \oand Y)_Z :(X\oand Y)(Z)\morph{} (X'\oand Y)(Z) :: C\mapsto C\cap S
\end{equation*}
and for $T:Y\morph{}Y'$ there is a natural transformation with components
\begin{equation*}
  (X \oand T)_Z :(X\oand Y)(Z)\morph{} (X\oand Y')(Z) :: C\mapsto C\cap T
\end{equation*}
These natural transformations commute, $(S\oand Y')_Z(X\oand T)_Z = (X'\oand T)_Z(S\oand Y)_Z$ and we can define $(S\oand T)$ to be given by their composition.

\begin{lem}\label{prop:oand_nat}
  $(S \oand Y)$ and $(X\oand T)$ are natural transformations with $(S\oand Y')_Z(X\oand T)_Z = (X'\oand T)_Z(S\oand Y)_Z$.
\end{lem}
\begin{proof}
  Proof given in appendix \ref{proof:oand_nat}.
\end{proof}

\begin{lem}\label{prop:oand_functor}
  The assignment $(X,Y)\mapsto (X\oand Y)(-)$ and $(S,T)\mapsto (S\oand T)$ gives a functor $\Slice\times\Slice\morph{}[\opcat{\Slice},\set]$.
\end{lem}
\begin{proof}
  Proof given in appendix \ref{proof:oand_functor}.
\end{proof}

We are now in a position to prove the main result of this section:
\begin{thm}\label{prop:sliceispro}
  $\Slice$ is a symmetric promonoidal category where the tensor is given above and the unit presheaf is given by $I(Z):= \mathcal{P}(\cat{C}[Z,Z])$.
\end{thm}
\begin{proof}
  Proof given in appendix \ref{proof:sliceispro}.
\end{proof}

Now we know that $\Slice$ is promonoidal under intersection, we will study when the presheaves assigned by this tensor are representable.
This allows us to ascertain where $\oand$ acts like a standard monoidal product on $\Slice$ and where it is possible for us to consider the tensor of slices to be another slice.

\begin{thm}
  When $X$ and $Y$ are jointly spacelike slices, the presheaf $(X\oand Y)(-)$ is representable.
\end{thm}
\begin{proof}
    Suppose $X$ and $Y$ are jointly spacelike.
    Note $\cat{C}[Z,X]\cap\cat{C}[Z,Y] \supseteq \cat{C}[Z,X\cap Y]$.
    Suppose there exists $\gamma\in \cat{C}[Z,X]\cap\cat{C}[Z,Y]$ with $\gamma\notin\cat{C}[Z,X\cap Y]$.
    Then $\gamma$ must pass through some $x\in X\backslash Y$ and some $y\in Y\backslash X$ but this would imply that $X$ and $Y$ are not jointly spacelike.
    Thus $\gamma$ cannot exist and it follows that $(X\oand Y)(Z) = \mathcal{P}(\cat{C}[Z,X\cap Y]) = \Slice(Z,X\cap Y) = \yo_{X\cap Y}(Z)$, noting that $X\cap Y$ is a slice because $X\cap Y\subseteq X$ and thus is an object of $\Slice$.
\end{proof}

In particular, the previous theorem shows that on jointly spacelike slices $\oand$ acts like intersection and we can make the identification $(X\oand Y)(-)\simeq X\cap Y$.
On the other hand, when the slices are not jointly spacelike there is no representative for $(X\oand Y)(-)$.
To show this we need the following lemma:

\begin{lem}\label{lem:tinycurves}
  Let $A\subseteq \mathcal{M}$ be a closed subset of $\mathcal{M}$.
  Then for any $x\in\mathcal{M}$, $x\notin A$, there exists a causal curve through $x$ which does not intersect $A$.
\end{lem}
\begin{proof}
  The timelike vector field is non-vanishing on $\mathcal{M}$ and as a result there must be a causal curve $\gamma$ through $x$.
  In a sufficiently small neighbourhood $U$ of $x$, $\gamma$ must restrict to a causal curve which is contained entirely within $U$.
  Since $A$ is closed and $\mathcal{M}$ is Hausdorff, this neighbourhood can be made sufficiently small such that $U\cap A=\varnothing$.
\end{proof}

\begin{thm}\label{thm:timelike_intersection}
  When $X$ and $Y$ are not jointly spacelike, the presheaf $(X\oand Y)(-)$ is not representable.
\end{thm}
\begin{proof}
  We make much use of Lemma \ref{lem:tinycurves}.
  Suppose $X$ and $Y$ are not jointly spacelike and suppose for a contradiction that $(X\oand Y)(-) = \Slice(-,Z)$ for some slice $Z$.

  Now suppose there exists a $z\in Z$ such that $z\notin X\cup Y$.
  We can find a causal curve $\gamma$ through $z$ that does not also pass through $X \cup Y$.
  It follows that $\gamma \in \Slice(Z,Z)$, but $\gamma\notin(X\oand Y)(Z)$.
  So $Z$ cannot represent the presheaf and we conclude $Z\subseteq X\cup Y$.

  Now take a $x\in X\backslash Y$.
  There exists a causal curve $\gamma$ passing through $x$ but not $Y$.
  Suppose that $x\in Z$, then $\gamma \in \Slice(Z,Z)$, but $\gamma\notin(X\oand Y)(Z)$.
  So $x\notin Z$.

  A similar argument shows that any $y\in Y\backslash X$ cannot be in $Z$ and thus $Z \subseteq X\cap Y$.

  Since $X$ and $Y$ are not jointly spacelike, $X\cup Y$ is not spacelike and there exists a causal curve $\gamma$ from $X\cup Y$ to itself.
  In particular $\gamma$ must pass through a point of $X$ and a point of $Y$, and not, say, through two points of $X$, since $X$ and $Y$ are slices.
  Then we would have $\gamma\in (X\oand Y)(X)$ but $\gamma\notin \Slice(X,Z)$ because if $\gamma\in \Slice(X,Z)$ it would pass through $X$ and $X\cap Y\subseteq X$, a contradiction with $X$ being a slice.
\end{proof}

So we have shown that $(X\oand Y)(-)$ is representable if and only if $X$ and $Y$ are jointly spacelike.
Note that one \textbf{cannot} define a partially monoidal category by just working with $\oand$ where it is representable because the unit presheaf is not representable (the whole manifold is not a slice) and therefore there is no unit object available in $\Slice$.

\section{The Structure of \sf Slice \bf and \sf Space \bf under Union}\label{sec:slice_union}
Let us now consider the structure of $\Slice$ and $\Space$ under union of slices and sets of curves.
The larger category $\Space$ where the objects are arbitrary subsets of the manifold $\mathcal{M}$ and the homsets are powersets of causal curves is a premonoidal category:
\begin{prop}
  $\Space$ is a strict premonoidal category under the operation of taking the union of regions and curves.
\end{prop}
\begin{proof}
  For objects $X$ and $Y$ assign them the object $X\otimes Y:= X\cup Y$.
  The assignment $(T:Y\morph{}Y') \mapsto (\cat{C}[X]\cup T: X\cup Y\morph{} X\cup Y')$ gives a functor $X\rtimes-:\cat{C}\morph{}\cat{C}$ because
  \begin{equation*}
    X\rtimes 1_Y = \cat{C}[X]\cup\cat{C}[Y] = \cat{C}[X\cup Y] = 1_{X\cup Y}
  \end{equation*}
  \begin{equation*}
    (X\rtimes f')(X\rtimes f) = (\cat{C}[X]\cup f')\cap (\cat{C}[X]\cup f) = \cat{C}[X] \cup (f'\cap f) = X\rtimes f'f
  \end{equation*}
  Similarly the assignment $(S:X\morph{}X') \mapsto (S\cup \cat{C}[Y])$ extends to a functor $-\ltimes Y:\cat{C}\morph{}\cat{C}$.
  The unit object is $I := \varnothing$ and the unit and associativity isomorphisms are identities, which it is straightforward to check are central.
\end{proof}

The above has a clear issue - $X\cup Y$ is generally not another slice and thus not an object of $\Slice$.
This means $\Slice$ cannot form a premonoidal category under union and we need to search for something that combines both premonoidal and promonoidal structures together.

There is no obstacle to defining presheaves $(X\oor Y)(-):\opcat{\Slice}\morph{}\set$ which send a slice $Z$ to the powerset of causal curves through $Z$ and either $X$ or $Y$:
\begin{equation*}
  (X\oor Y)(Z) := \mathcal{P}(\cat{C}[Z,X]\cup\cat{C}[Z,Y])
\end{equation*}
On morphisms $S:Z'\morph{}Z$ this presheaf acts by intersection:
\begin{equation*}
  (X\oor Y)(S) :(X\oor Y)(Z) \morph{} (X\oor Y)(Z') :: C\mapsto C\cap S
\end{equation*}

\begin{lem}\label{prop:or1}
  $(X\oor Y)(-)$ is a presheaf.
\end{lem}

Similarly, there is no obstacle to defining natural transformations acting on either the left or right of $\oor$.
For $S:X\morph{}X'$ there is a natural transformation with components
\begin{equation*}
  (S \oor Y)_Z :(X\oor Y)(Z)\morph{} (X'\oor Y)(Z) :: C\mapsto C\cap (S\cup \cat{C}[Y])
\end{equation*}
and for $T:Y\morph{}Y'$ there is a natural transformation with components
\begin{equation*}
  (X \oor T)_Z :(X\oor Y)(Z)\morph{} (X\oor Y')(Z) :: C\mapsto C\cap (\cat{C}[X] \cup T)
\end{equation*}

\begin{lem}\label{prop:or2}
  $(S \oor Y)$ and $(X \oor T)$ are natural transformations.
\end{lem}

What fails in comparison to $\oand$ is that, in general, the components of these natural transformations do not obey the interchange law, so we cannot hope that these data give a functor $\Slice\times\Slice\morph{}[\opcat{\Slice},\set]$.
Nevertheless, the natural transformations are functorial on each side of the tensor and it is easy to verify that the assignment does give a functor $\oor:\Slice\funny\Slice\morph{}[\opcat{\Slice},\set]$ where $\funny$ is the funny tensor product of categories.

\begin{lem}\label{prop:oor_functor}
  The data of Lemmas \ref{prop:or1} and \ref{prop:or2} specify a functor $\Slice\funny\Slice\morph{}[\opcat{\Slice},\set]$
\end{lem}
\begin{proof}
  Proof given in appendix \ref{proof:oor_functor}.
\end{proof}

In this way $\Slice$ seems to combine both the structures of premonoidal and promonoidal categories.
We leave it as future work to make rigorous the associativity and unitality of this structure but we note that the representable presheaf at the empty slice $\yo_\varnothing$ is likely the unit of a suitably defined structure.

Similarly to the intersection case we can study when the presheaves $(X\oor Y)(-)$ are representable:

\begin{thm}
  When $X$ and $Y$ are jointly spacelike, the presheaf $(X\oor Y)(-)$ is representable.
\end{thm}
\begin{proof}
  Suppose $X$ and $Y$ are jointly spacelike.
  Then $(X\oor Y)(Z) = \mathcal{P}(\cat{C}[Z,X]\cup\cat{C}[Z,Y]) = \mathcal{P}(\cat{C}[Z,X\cup Y]) = \yo_{X\cup Y}(Z)$ where we have used the fact that $X\cup Y$ is spacelike and thus an object of $\Slice$.
\end{proof}

\begin{thm}
  When $X$ and $Y$ are not jointly spacelike, the presheaf $(X\oor Y)(-)$ is not representable.
\end{thm}
\begin{proof}
  We make use of Lemma \ref{lem:tinycurves}.
  Suppose $X$ and $Y$ are not jointly spacelike and suppose for a contradiction that $(X\oor Y)(-)=\Slice(-,Z)$ for some slice $Z$.
  By the same argument made in the proof of Theorem \ref{thm:timelike_intersection} we must have $Z\subseteq X\cup Y$.

  Since $X$ and $Y$ are not jointly spacelike, $X\cup Y$ is not spacelike and thus there exists a causal curve $\gamma$ connecting two points of $X\cup Y$.
  It must be the case that one of these points is in $X\backslash Y$ and the other in $Y\backslash X$ else $X$ or $Y$ could not be slices.
  Write $x\in X\backslash Y$ and $y\in Y\backslash X$ for these points that $\gamma$ passes through and note that they can be the only points of $X\cup Y$ that $\gamma$ intersects else $X$ or $Y$ could not be slices.

  Now note that $\gamma$ restricts to a causal curve $\gamma_x$ which passes through $x$ but not $y$ and similarly a causal curve $\gamma_y$ which passes through $y$ but not $x$.

  Suppose that $x\notin Z$, then $\gamma_x\in(X\oor Y)(X)$ but $\gamma_x\notin \Slice(X,Z)$, noting that $Z\subseteq X\cup Y$ so that $\gamma_x$ intersects $Z$ at only $x$.
  So we conclude that $x\in Z$.

  Similarly, suppose that $y\notin Z$, then $\gamma_y\in(X\oor Y)(Y)$ but $\gamma_y\notin \Slice(Y,Z)$.
  So we conclude that $y\in Z$.

  We see that $\gamma$ is a causal curve connecting two distinct points of $Z$ and consequently $Z$ cannot be a slice.
\end{proof}

So we have shown that the presheaf $(X\oor Y)(-)$ is representable if and only if $X$ and $Y$ are jointly spacelike.
By restricting $\oor$ to these slices we can recover a partial premonoidal structure on $\Slice$ by defining the tensor to be given by the representative.
The unit of this partial premonoidal category is the empty slice $\varnothing$.

Now that we have two tensor-like structures on $\Slice$ we would like to know how they interact.
Given that $\oor$ behaves like union and $\oand$ like intersection, it seems reasonable to expect some sort of distributivity between them.
To understand this at the level of the profunctors we require the following definition:
\begin{defn}[Multiplicative Kernel \cite{day_fourier}]
  Let $(\cat{C},P,I)$ and $(\cat{D},Q,J)$ be promonoidal categories.
  A multiplicative kernel is a profunctor $K:\cat{C}\pmorph\cat{D}$ such that
  \begin{equation*}
    Q(K\times K) \cong KP \hspace{2cm} KI \cong J
  \end{equation*}
  where concatenation is profunctor composition.
\end{defn}

\begin{remark}
  Viewing $\cat{C}$ and $\cat{D}$ as pseudomonoids in $\Prof$, a multiplicative kernel is a homomorphism of these monoids.
\end{remark}

Each slice $X$ determines an endoprofunctor $(X\oor -)(-):\Slice\pmorph \Slice$ and it is the case that each of these is a multiplicative kernel for $\Slice$ equipped with $\oand$.

\begin{thm}\label{thm:kernel}
  For every slice $X$, $(X\oor -)(-)$ is a multiplicative kernel for $(\Slice,\oand)$.
\end{thm}
\begin{proof}
  Proof given in appendix \ref{proof:kernel}.
\end{proof}

\section{Interpreting the operations in $\Slice$}\label{sec:logic}

We have shown that $\Slice$ admits two operations $\oor$ and $\oand$ taking a pair of spacelike slices to a ``generalised'' slice, i.e. a presheaf over slices. Here, we give (the beginnings of) a physical interpretation for these operations.

First, it is helpful to shift from thinking geometrically about slices to thinking logically about them. That is, we can think of a slice $X$ as a logical predicate, namely that a system satisfies a certain property at a certain moment in time. The simplest non-trivial example is 1+1 dimensional Minkowski space, where a particle in 1D space traces out a causal curve through $\mathbb R^{1+1}$.

As a simple case, we can consider slices of the form $X := \{t_1\} \times P$ for a time $t_1 \in \mathbb R$ and a closed subset $P \subseteq \mathbb R$.
We can now think of $P$ as saying something about the position of a particle at time $t_1$, e.g. ``the particle's position is $\geq x_1$''. Similarly, another slice $Y := \{t_2\} \times Q$, expresses that a certain property $Q$ holds for a particle at time $t_2$, e.g. ``the particle's position is $\leq x_2$''.

We can now think about whether it makes sense to take conjunctions or disjunctions of these kinds of predicates. If $t_1 = t_2$, then everything works out exactly as one would expect. Namely, $X \oand Y = \{t_1\} \times (P \cap Q)$, which captures the statement that at time $t_1$, ``the particles position is $\geq x_1$ AND it is $\leq x_2$''. Similarly, $X \oor Y = \{t_1\} \times (P \cup Q)$, capturing the OR if predicates $P$ and $Q$ at a fixed time $t_1$.

If we look at arbitrary pairs of jointly spacelike slices $X$ and $Y$, then much the same interpretation holds, but rather than separating the time and space coordinates in a fixed reference frame, we can regard $X$ and $Y$ as living on the same spacelike hypersurface.

The more interesting case is of course when $X$ and $Y$ are not jointly spacelike. While we can't make sense of $X \oand Y$ and $X \oor Y$ as spacelike slices themselves, we can make sense of them relative to a third, ``probe'' slice $Z$. If we restrict to the simpler case where $X = \{t_1\} \times P$ and $Y = \{t_2\} \times Q$, now with $t_1 \neq t_2$ and possibly some causal curves between $X$ and $Y$, then any $S \in (X \oand Y)(Z)$ is a set of causal curves that first passes through $Z$ then must satisfy $P$ at $t_1$ AND $Q$ at $t_2$. Hence, $\oand$ captures conjunction, but with predicates at different times. Similarly, $\oor$ captures this generalised kind of disjunction.

We can apply this kind of interpretation to arbitrary pairs of slices $X, Y$, not just those which take a product form in a fixed reference frame, however the meaning is slightly less intuitive in some cases, like when $X$ and $Y$ intersect and are furthermore not jointly spacelike. Nevertheless, we obtain a notion of conjunction and disjunction which is defined everywhere, and thanks to Theorem~\ref{thm:kernel}, distributes as one would expect. Hence, we have the beginnings of a logic for (generalised) spacetime slices. However there is much still to explore. For example, there is no clear ``universal'' notion of negation here, but one may be able to negate a slice relative to another one, e.g. some Cauchy surface containing the slice.

\section{Conclusion and Future Work}

We have shown that the category $\Slice$ of spacelike slices and causal curves admits two generalised tensor-like structures, corresponding to conjunction and disjunction. We see several avenues of future work. One is the complete characterisation of the structure $\oor$ defined in Section~\ref{sec:slice_union}, which combines elements of both a premonoidal and promonoidal product. As promonoidal and strict premonoidal categories can be formalised as pseudomonoids in a suitable monoidal bicategory, one might hope to do similar for the ``pre-promonoidal'' structure $\oor$.

As hinted at the end of the previous section, there seems to be much more left to say about the logical interpretation of connectives in $\Slice$. For instance, one could try to obtain an analogue to full classical logic by introducing a (suitably localised) negation. It also seems natural to study non-commutative connectives such as the ``sequence'' product $\triangleleft$ present in the logic BV~\cite{Guglielmi2007}, which was recently shown to capture the \textit{one-way signalling} processes in the $\textrm{Caus}[-]$ construction~\cite{SimmonsKissinger2022}.

Another direction is to investigate other places tensor-like structures appear, particularly within models that have some notion of ``causality'' which may be different from the usual relativistic one. For example, by imposing restrictions on the Petri nets of \cite{baez_petri}, one may force the monoidal category $\mathsf{FP}$ developed there to be only promonoidal. In such a case it seems that the fibres $\mathsf{FP}_i$ are no longer premonoidal but can be described by a pre-promonoidal category.

While the category $\Slice$ we defined here gives an interesting toy theory for exploring spacetime, causal curves, and associated notions of logic and compositionality, it is by no means the ``one true'' category of spacetime. It would be interesting to study variations on this structure, which may have different, possibly more natural notions of composition. For example, instead of intersecting sets of curves, one could define a category $\Slice'$ where composition is given by ``gluing'' curves together, somewhat in the same spirit as~\cite{gogioso_functorial}. Such a category seems more amenable to an alternative view of AQFT, as functors $\Slice'\morph{}\mathsf{Alg}_k$, or indeed with codomain taken to be any reasonable process theory.

Finally, one could compare our approach to other categorical models of causality and spacetime, such as the formulation using idempotent subobjects~\cite{moliner_thesis,moliner_space}, the order-theoretic formulation of~\cite{gogioso_functorial}, and the aforementioned $\textrm{Caus}[-]$ construction~\cite{Kissinger2019a,SimmonsKissinger2022}.

\subsubsection*{Acknowledgements}
JH is supported by University College London and the EPSRC [grant number EP/L015242/1].
AK is supported by Grant No. 61466 from the John Templeton Foundation as part of the QISS project. The opinions expressed in this publication are those of the authors and do not necessarily reflect the views of the John Templeton Foundation.
Many thanks to Matt Wilson, Cole Comfort, John van de Wetering, Bob Coecke, and Raymond Lal for useful recent (and not-so-recent!) discussions on categories of spacetime; and to Stefano Gogioso for noticing a few issues with an early draft of this work.

\bibliographystyle{eptcs}
\bibliography{bibliography}

\appendix

\section{Partially Monoidal Categories as Promonoidal Categories}\label{sec:partiallymonoidal}
In this appendix we compare the partially monoidal categories of \cite{coecke_causalcats,lal_causal,gogioso_church} to promonoidal categories.
We discuss a class of partially monoidal categories that can be equivalently described as promonoidal categories which are representable wherever the presheaves are non-empty and discuss when it is possible to derive a partially monoidal category from a promonoidal one.

\begin{defn}[Partial Functor \cite{gogioso_church}]
  A partial functor $\cat{C}\rightharpoonup\cat{D}$ is a span of functors $\cat{C}\lmorph{i} \cat{S} \morph{F}\cat{D}$ where $i$ is an opisofibration, embedding $S$ as a subcategory of $\cat{C}$ (so $i$ is full, faithful and $\cat{S}$ is a replete subcategory of $\cat{C}$).
  Composition of partial functors is by pullback.
  A morphism of partial functors $(\phi,\eta):(i,F)\morph{}(j,G)$ is a pair of a functor $\phi:\cat{S}\morph{}\cat{S'}$ between the apexes of the spans and a natural transformation $\eta:F\implies G\phi$,
  \begin{equation}\label{eq:pfunmorph}
    \begin{tikzcd}
      & \cat{S} \ar[d,"\phi"] \ar[ddl,"i"', bend right, ""{name=U, below, sloped}] \ar[ddr,"F", bend left, ""{name=V,below,sloped}] & \\
      & \cat{S'} \ar[dl,"j"] \ar[dr,"G"'] \ar[from=V,Rightarrow,"\eta"'{xshift=0.1cm,yshift=0cm}] & \\
      \cat{C} & & \cat{D}
    \end{tikzcd}
  \end{equation}
  Categories, partial functors and morphisms of partial functors form a monoidal bicategory $\pcat$ where the tensor is given pointwise by taking the product of categories and the product of the underlying functors in the spans.
  Note that full and faithful opisofibrations are closed under composition and stable under pullback.
\end{defn}

\begin{defn}[Partially Monoidal Category \cite{gogioso_church}]
  A category $\cat{C}$ is partially monoidal if it is equipped with:
  \begin{itemize}
    \item A partial tensor product functor $\boxtimes:\cat{C}\times\cat{C}\rightharpoonup\cat{C}$
    \item A unit object $I$
  \end{itemize}
  together with associativity and unit natural isomorphisms such that the triangle and pentagon equations hold.
\end{defn}

\begin{remark}
  A very concise definition of a partially monoidal category $\cat{C}$ is as a pseudomonoid in $\pcat$
\end{remark}

It may not be immediately apparent that there are connections between partially monoidal and promonoidal categories.
It turns out though that there is a case where the two coincide on-the-nose.

There exists a special class of partial functors where the left leg is not only an opisofibration but a proper discrete opfibration.
This makes the left leg a \textit{cosieve} which coincides with the definition of partial functor given by \cite{benabou}.
Demanding that the left leg is a cosieve ensures that the subcategory on which the tensor is defined is closed under post-composition with morphisms of $\cat{C}\times\cat{C}$.
This captures the following physical intuition: if $X\otimes Y$ exists and there is a morphism $X\morph{}X'$ then $X'\otimes Y$ exists too.
Thus we maintain the intuition that if one applies a local map to $X$ then the tensor product should still exist afterwards.
From a mathematical perspective, when the left leg of the tensor product partial functor is a cosieve, the partially monoidal category is equivalent to a promonoidal one.
Indeed, B\'enabou notes that there is an 1-1 correspondence (up to isomorphism) between partial functors with left leg a cosieve and profunctors which factorise through the representable and empty presheaves \cite{benabou}.
In this light the following proposition is not surprising but there is a little effort required in checking that everything works out:

\begin{prop}\label{prop:partpro}
  A partially monoidal category $(\cat{C},\boxtimes,J)$ whose left leg of the tensor product partial functor is a cosieve is a promonoidal category with representable unit and a tensor $\otimes(-,b,c)$ which is either representable or empty for each $(b,c)\in\cat{C}\times\cat{C}$.
\end{prop}

\begin{proof}
  Proof given in appendix \ref{proof:partpro}.
\end{proof}

There are many examples of partially monoidal categories which are not equivalent to promonoidal ones and vice-versa.
For instance, we require that the unit presheaf $J(-)$ of a promonoidal category is representable to have any hope that it is a partially monoidal category.

Conversely, one might hope (similarly to monoidal categories) that all partially monoidal categories could be turned into promonoidal ones.
In general this is not possible though as taking the representable presheaves at the defined points of the partial tensor is not enough to define a profunctor $\cat{C}\times\cat{C}\pmorph\cat{C}$.
Indeed, a promonoidal category still has a total tensor, just into the presheaf category,

It is possible though to derive partially monoidal structures from a promonoidal one with representable unit presheaf $J(-)$, by pulling back the promonoidal tensor along the Yoneda embedding whenever it is representable.
There is of course a canonical ``maximal'' such partially monoidal structure induced by defining it everywhere it is possible to do so, i.e. everywhere the promonoidal tensor is representable.

One may wonder if there are any further connections between partial functors and profunctors - is there a category that unites them?
This would allow us to place the two on equal footing and compare arbitrary partially monoidal and promonoidal categories.
The key to this unification is the following result:
\begin{thm}[\cite{benabou,loregian_fibration}]
  There is an equivalence of categories between profunctors $\cat{C}\pmorph\cat{D}$ and two-sided discrete fibrations $\mathsf{DFib}(\cat{C},\cat{D})$.
\end{thm}

A two-sided discrete fibration is a span of functors $\cat{C}\lmorph{F} \cat{E}\morph{G}\cat{D}$ where:
\begin{itemize}
  \item each $F(e)\morph{}c'$ in $\cat{C}$ has unique lift $f:e\morph{}e'$ in $\cat{E}$ such that $G(f)=1_{G(e)}$,
  \item each $d\morph{}G(e)$ in $\cat{D}$ has unique lift $g:e'\morph{}e$ in $\cat{E}$ such that $F(g)=1_{F(e)}$,
  \item for each $f:e\morph{}e'$ in $\cat{E}$, the codomain of the lift of $Ff$ equals the domain of the lift of $Gf$, and their composite is $f$.
\end{itemize}

The two-sided discrete fibration corresponding to a profunctor $P:\cat{C}\pmorph\cat{D}$ is given by the projections out of the category $\mathsf{Sec}(P)$ of sections of the collage of $P$.
The objects of $\mathsf{Sec}(P)$ are the elements of the sets $P(d,c)$ for all $c$ and $d$.
A morphism $x\in P(d,c) \morph{} x'\in P(d',c')$ is given by a pair of arrows $f$ and $g$ such that $P(g,1)(x')=P(1,f)(x)$.

Consequently, each profunctor has a canonical span and by working in the category of spans of functors one can study the partial functors and profunctors side-by-side.
For instance, suppose $(\cat{C},\otimes,J)$ is a promonoidal category with $J(-)\cong\cat{C}(-,I)$.
There is a partial monoidal structure $(\boxtimes,I)$ on $\cat{C}$ given by pulling back $\otimes$ along the Yoneda embedding whenever it is representable - that is, whenever $\otimes(-,b,c)\cong\cat{C}(-,x_{bc})$ for some objects $b$ and $c$, we define $b\boxtimes c := x_{bc}$.
Write $\overline{\cat{C}\times\cat{C}}$ for the subcategory of $\cat{C}\times\cat{C}$ where the promonoidal tensor is representable.
Then there is a 2-cell in $\mathsf{Span}(\Cat)$ capturing the extension of the partially monoidal structure on $\cat{C}$ to the promonoidal structure:

\begin{equation*}
  \begin{tikzcd}
    & \overline{\cat{C}\times\cat{C}} \ar[d,"\phi"] \ar[ddl,"i"', bend right, ""{name=U, below, sloped}] \ar[ddr,"\boxtimes", bend left, ""{name=V,below,sloped}] & \\
    & \mathsf{Sec}(\otimes) \ar[dl,"p_1"] \ar[dr,"p_0"'] & \\
    \cat{C}\times\cat{C} & & \cat{C}
  \end{tikzcd}
\end{equation*}
where $\phi$ sends $(b,c)$ to $1_{b\boxtimes c,b\boxtimes c}\in \otimes(b\boxtimes c, b, c)$ and $(g,f)$ to $(g\boxtimes f, g, f)$.

\section{Proofs}

\subsection{Proof of Proposition \ref{prop:sliceprocopro}}\label{proof:sliceprocopro}
\begin{proof}
  The projections are given by
  \begin{align*}
    \pi_0 = \cat{C}[X]&:X\cup Y\morph{}X \\
    \pi_1 = \cat{C}[Y]&:X\cup Y\morph{}Y
  \end{align*}
  while the coprojections are given by
  \begin{align*}
    i_0 = \cat{C}[X]&:X\morph{}X\cup Y \\
    i_1 = \cat{C}[Y]&:Y\morph{}X\cup Y
  \end{align*}
  Given $f:Z\morph{}X$ and $f':Z\morph{}Y$, the universal arrow completing the product diagram is $\langle f,f'\rangle =f\cup f':Z\morph{}X\cup Y$, and given $g:X\morph{}Z$ and $g':Y\morph{}Z$, the universal arrow completing the coproduct diagram is $[g,g'] = g\cup g':X\cup Y\morph{} Z$.
  Indeed, it follows that the diagrams commute because $X$ and $Y$ are jointly spacelike with $X\cap Y=\varnothing$ and thus $f\cap\cat{C}[Y] = f'\cap\cat{C}[X] = g\cap\cat{C}[Y]= g'\cap\cat{C}[X] = \varnothing$.
\end{proof}

\subsection{Proof of Lemma \ref{prop:oand_nat}}\label{proof:oand_nat}
\begin{proof}
  Note that the following diagram commutes for any $U:Z'\morph{}Z$
  \begin{equation*}
    \begin{tikzcd}
      (X\oand Y)(Z) \ar[r, "(X\oand Y)(U)"] \ar[d, "(S\oand Y)_Z"'] & (X\oand Y)(Z') \ar[d, "(S\oand Y)_{Z'}"] \\
      (X'\oand Y)(Z) \ar[r, "(X'\oand Y)(U)"'] & (X'\oand Y)(Z')
    \end{tikzcd}
  \end{equation*}
  because on the top path we see $C\mapsto C\cap U \mapsto (C\cap U)\cap S$ while on the bottom path $C\mapsto C\cap S \mapsto (C\cap S)\cap U$.
  Naturality of $(X\oand T)$ follows similarly and checking the commutativity condition is straightforward.
\end{proof}

\subsection{Proof of Lemma \ref{prop:oand_functor}}\label{proof:oand_functor}
\begin{proof}
  Firstly note that each component of $1_X\oand 1_Y : (X\oand Y)(-) \implies (X\oand Y)(-)$ is just the identity.
  Thus it is the identity natural transformation and we conclude $1_X\oand 1_Y = 1_{(X\oand Y)(-)}$.

  Now take $S:X\morph{}X'$ and $S':X'\morph{}X''$.
  The arrow $(S'\oand Y)_Z \circ (S\oand Y)_Z$ acts as $C\mapsto (C\cap S)\cap S'$
  while the arrow $((S'\circ S) \oand Y)_Z$ acts as $C\mapsto C\cap (S'\cap S)$.
  Thus the components of the composite natural transformation $(S'\oand Y) \circ (S\oand Y)$ equal those of $((S'\circ S) \oand Y)$.

  A similar argument holds for arrows $T:Y\morph{}Y'$ and because $(S \oand Y)$ and $(X\oand T)$ commute we are done.
\end{proof}

\subsection{Proof of Theorem \ref{prop:sliceispro}}\label{proof:sliceispro}
\begin{proof}
  Let us begin with associativity $\oand(\oand \times 1) \cong \oand (1\times \oand)$.
  Note that by Yoneda we have
  \begin{align*}
    \oand (\oand\times 1)(W,X,Y,Z) & = \int^{A,B} \oand(W,A,B) \times \oand(A,X,Y) \times \Slice(B,Z) \\
    & \cong \int^{A} \oand(W,A,Z) \times \oand(A,X,Y) \label{eq:coendboi}
  \end{align*}
  While
  \begin{equation*}
    \oand (1\times \oand)(W,X,Y,Z) \cong \int^{A} \oand(W,X,A) \times \oand(A,Y,Z)
  \end{equation*}

  Let us show there is a canonical identification $\oand (\oand\times 1)(W,X,Y,Z) \cong \mathcal{P}(\cat{C}[W,X]\cap\cat{C}[W,Y]\cap\cat{C}[W,Z])=: \Lambda$.
  There are functions
  \begin{equation*}
    \oand(W,A,Z) \times \oand(A,X,Y) \morph{} \Lambda :: (S,T)\mapsto S\cap T
  \end{equation*}
  which form a cowedge with apex $\Lambda$.
  By the universal property of the coend this induces a unique function $g:\int^A\oand(W,A,Z) \times \oand(A,X,Y) \morph{} \Lambda$ making the obvious cowedge diagrams commute.

  We can also construct a function $f$ by composing
  \begin{equation*}
    \Lambda \morph{f'} \oand(W,W,Z) \times \oand(W,X,Y) \morph{\text{copr}_W} \int^A\oand(W,A,Z) \times \oand(A,X,Y)
  \end{equation*}
  where $f'$ acts as $S\mapsto (S,S)$.

  The universal property of the coend implies that the composition $fg = 1$, or we can check explicitly:
  \begin{equation*}
    (S,T)\mapsto S\cap T \mapsto (S\cap T,S\cap T)
  \end{equation*}
  upon which we simply need to note that we have $(S,T)=(S\cap S,T\cap T) \sim (S\cap T,S\cap T)$.

  Similarly, it is straightforward to show that $gf=1$: $S\mapsto (S,S)\mapsto S\cap S=S$.
  Thus $\Lambda \cong \int^A\oand(W,A,Z) \times \oand(A,X,Y)$ as sets.

  Now note that this isomorphism is in fact natural in $W,X,Y$ and $Z$.
  Let $w:W'\morph{}W, x:X\morph{}X', y:Y\morph{}Y', z:Z\morph{}Z'$, then we have
  \begin{equation*}
    \begin{tikzcd}
      (S,T) \ar[r,mapsto] \ar[d,"g_{WXYZ}"',mapsto] & (S\cap w\cap z,T\cap x\cap y) \ar[d,"g_{W'X'Y'Z'}",mapsto] \\
      S\cap T \ar[r,mapsto] & S\cap T\cap w\cap x\cap y \cap z
    \end{tikzcd}
  \end{equation*}
  Thus exhibiting the desired natural isomorphism.

  A similar argument shows that $\oand (1\times \oand)(W,X,Y,Z) \cong \Lambda$, and thus we have established the associativity natural isomorphism.

  The pentagon equation is given by (writing $i$ for the interchange and ignoring the associativity isomorphisms of profunctor composition):
  \begin{equation*}
    \begin{tikzcd}
      & \oand^a_{xe} \oand^x_{yd} \oand^y_{bc} \ar[rr,"1\circ\alpha"] \ar[ld,"\alpha\circ1"'] & & \oand^a_{xe} \oand^x_{by} \oand^y_{cd} \ar[rd,"\alpha\circ1"] & \\
      \oand^a_{yx} \oand^x_{de} \oand^y_{bc} \ar[rrd,"{(\alpha\circ1)i}"'] & & & &\oand^a_{bx} \oand^x_{ye} \oand^y_{cd} \\
      & & \oand^a_{bx} \oand^x_{cy} \oand^y_{de} \ar[rru,leftarrow,"1\circ\alpha"'] & &
    \end{tikzcd}
  \end{equation*}
  Clockwise we have the following mapping:
  \begin{equation*}
    (S,T,V) \mapsto (S,T\cap V,T\cap V) \mapsto (S\cap T\cap V,S\cap T\cap V, T\cap V) \mapsto (S\cap T\cap V,S\cap T\cap V, S\cap T\cap V)
  \end{equation*}
  while anticlockwise we have
  \begin{equation*}
    (S,T,V) \mapsto (S\cap T,S\cap T, V) \mapsto (S\cap T\cap V,S\cap T, S\cap T\cap V)
  \end{equation*}
  and it clear that $(S\cap T\cap V,S\cap T, S\cap T\cap V)\sim (S\cap T\cap V,S\cap T\cap V, S\cap T\cap V)$ under the coend equivalence relation.
  Thus the pentagon commutes.

  Now we show the existence of the unit isomorphisms $\oand (I\times 1) \cong 1 \cong \oand (1\times I)$.

  Much of the construction is similar to the previous argument, so we leave the reader to fill in some of the details.
  There exist functions $\oand(-,=,B)\times\mathcal{P}(\cat{C}[B,B]) \morph{} \Slice(-,=)$ for each $B$ given by sending $(S,T)\mapsto S\cap T$.
  These functions form a cowedge and therefore induce a unique function $\int^B \oand(-,=,B)\times\mathcal{P}(\cat{C}[B,B]) \morph{} \Slice(-,=)$.

  The inverse of this function is given by the function $S\mapsto (S,S)$ which factorises via $\text{copr}$.
  It is straightforward to check that these give the left unit natural isomorphism, and the construction of the right unit is similar.

  Writing $\yo$ for an application of the Yoneda lemma, the triangle equation is given by
  \begin{equation*}
    \begin{tikzcd}
      & \oand^a_{bc} & \\
      \oand^a_{xc}\oand^x_{by}I^y \ar[rr,"\alpha\circ 1"'] \ar[ur,"\yo\rho"] & & \oand^a_{bx}I^y\oand^x_{yc} \ar[ul,"\yo\lambda"']
    \end{tikzcd}
  \end{equation*}
  and it is little work to check that this commutes.

  The symmetry $(X\oand Y)(Z)\morph{}(Y\oand X)(Z)$ is given by the identity map for all $X,Y$ and $Z$.
\end{proof}

\subsection{Proof of Lemma \ref{prop:oor_functor}}\label{proof:oor_functor}
\begin{proof}
  Take $S:X\morph{}X'$ and $S':X'\morph{}X''$.
  Then $(S'\oor Y)_Z(S\oor Y)_Z$ acts as $C\mapsto C\cap(S\cup\cat{C}[Y])\cap(S'\cup\cat{C}[Y]) = C\cap((S\cap S')\cup \cat{C}[Y])$ which is precisely the same as the action of $(S'S\oor Y)_Z$.
  We conclude $(S'\oor Y)_Z(S\oor Y)_Z = (S'S\oor Y)_Z$.

  A similar argument shows that $(X\oor T')_Z(X\oor T)_Z = (X \oor T'T)_Z$ and thus we have functoriality of $(\mathord{-} \oor \mathord{=})$ in each component.
  This is enough to extend to functoriality from the funny tensor.
\end{proof}

\subsection{Proof of Theorem \ref{thm:kernel}}\label{proof:kernel}
\begin{proof}(Sketch).
  The proof is similar and uses the same methods as Theorem \ref{prop:sliceispro} so we only sketch the idea.

  Fix a slice $A$.
  We will show that $(A\oor -)(-)$ is a kernel.

  Starting with the units we need to show that $\int^X \oor^Z_{AX} J^X \cong J^Z$.
  There are functions $\oor^Z_{AX} J^X \morph{} J^Z$ sending $(S,T)\mapsto S\cap(T\cup\cat{C}[A])$.
  These are dinatural in $X$ and thus form a cowedge factorising uniquely via the coend.
  As a result we have a function $\int^X \oor^Z_{AX} J^X \morph{} J^Z$.
  This function is an isomorphism with inverse given by $S\mapsto (S,S)$ which factorises via $\text{copr}$.
  Indeed,
  \begin{equation*}
    S \mapsto (S,S) \mapsto S\cap (S\cup\cat{C}[A]) = S
  \end{equation*}
  and
  \begin{align*}
    (S,T)\mapsto S\cap(T\cup \cat{C}[A]) & \mapsto (S\cap(T\cup\cat{C}[A]), S\cap(T\cup\cat{C}[A]))\\
    & \sim (S\cap(S\cup\cat{C}[A]), T\cap T)\\
    & = (S,T)
  \end{align*}

  As for the multiplications we want to show $\int^Z \oor^W_{AZ}\oand^Z_{XY} \cong \int^{ZZ'} \oor^Z_{AX}\oor^{Z'}_{AY}\oand^W_{ZZ'}$ which it is easiest to do by showing each is naturally isomorphic to $\Lambda:= \mathcal{P}(\cat{C}[W,A]\cup(\cat{C}[W,X]\cap\cat{C}[W,Y]))$.
  For the former, there is a cowedge with components $(S,T)\mapsto S\cap(T\cup\cat{C}[A])$, with the inverse to the induced map given by $S\mapsto (S,S)$, as in the case of the units.
  For the latter, there is a cowedge with components $(S,T,V)\mapsto S\cap T\cap V$, with the inverse to the induced map given by $S\mapsto (S,S,S)$.

  To show that all the isomorphisms are natural is little work.
\end{proof}

\subsection{Proof of Proposition \ref{prop:partpro}}\label{proof:partpro}
\begin{proof}
  In a slight abuse of notation write $\cat{C}\times\cat{C} \lmorph{i}\cat{S}\morph{\boxtimes}\cat{C}$ for the underlying span of the partial functor $\boxtimes:\cat{C}\times\cat{C}\rightharpoonup\cat{C}$, and note that $J:1\rightharpoonup\cat{C}$ is simply a normal functor $J:1\morph{}\cat{C}$, in other words an object $J$ of $\cat{C}$.
  Just like for monoidal categories we can define a promonoidal structure on $\cat{C}$ by taking $(X\otimes Y)(Z) := \cat{C}(Z,X\boxtimes Y)$ whenever $(X,Y)\in \cat{S}$ and $(X\otimes Y)(Z) := \varnothing$ otherwise.
  The unit is the representable presheaf at $J$, $\yo_J$.

  The associativity isomorphism of a partially monoidal category induces the following arrows:
  \begin{equation}\label{cd:pcatass}
    \begin{tikzcd}
      & (\cat{S} \times \cat{C}) \times_{\cat{C}\times\cat{C}} \cat{S} \ar[ddl,"(i\times 1)\pi_0"', bend right] \ar[ddr,"\boxtimes\pi_1", bend left, ""{name=U, below, sloped}] \ar[d,"\phi"]& \\
      & (\cat{C} \times \cat{S}) \times_{\cat{C}\times\cat{C}} \cat{S} \ar[dl,"(1\times i)\pi_0"] \ar[dr,"\boxtimes\pi_1"', ""{name=V, below, sloped}] \ar[from=U, "\alpha"'{xshift=0.2cm,yshift=0cm}, Rightarrow] & \\
      \cat{C} \times \cat{C} \times \cat{C} & & \cat{C}
    \end{tikzcd}
  \end{equation}
  where $\pi_0$ and $\pi_1$ are the canonical projections from the pullback and $\alpha$ is a natural isomorphism.

  Given a cospan of functors $\cat{C}\morph{F}\cat{E}\lmorph{G}\cat{D}$, the pullback $\cat{C}\times_\cat{E}\cat{D}$ is the category consisting of pairs of objects $(c,d)$ with $Fc=Gd$ and pairs of morphisms $(f,g)$ with $Ff=Gg$.
  We can think of $(\cat{S} \times \cat{C}) \times_{\cat{C}\times\cat{C}} \cat{S}$ as the category with objects $(\, ((a,b),c),(a\boxtimes b, c)\, )$ where $(a,b)\in\cat{S}$ and $c\in\cat{C}$ with $(a\boxtimes b,c)\in\cat{S}$, while $(\cat{C} \times \cat{S}) \times_{\cat{C}\times\cat{C}} \cat{S}$ has objects $(\, (a,(b,c)), (a,b\boxtimes c)\, )$ where $(b,c)\in \cat{S}$ and $a\in\cat{C}$ with $(a,b\boxtimes c)\in\cat{S}$.
  The left triangle of \eqref{cd:pcatass} ensures that $\phi$ must act to send $(\, ((a,b),c),(a\boxtimes b, c)\, )\mapsto (\, (a,(b,c)), (a,b\boxtimes c)\, )$.
  The right triangle of \eqref{cd:pcatass} then implies that the components of $\alpha$ have type $\alpha_{a,b,c}:(a\boxtimes b)\boxtimes c \morph{} a\boxtimes (b\boxtimes c)$.
  This induces the necessary isomorphism $\otimes^a_{xd}\otimes^x_{bc}\morph{} \otimes^a_{bx}\otimes^x_{cd}$ and checking the pentagon coherence equation now follows the same standard proof as Theorem \ref{prop:monoidal_promonoidal}.

  The right unit isomorphism induces the following arrows:
  \begin{equation*}
    \begin{tikzcd}
      & (\cat{C} \times 1) \times_{\cat{C}\times\cat{C}} \cat{S} \ar[dl,"\pi_0"', bend right] \ar[ddr,"\boxtimes\pi_1", bend left, ""{name=U, below, sloped}] \ar[d,"\psi"]& \\
      \cat{C} \times 1 \ar[d,"\sim"'] & \cat{C} \ar[dl,"1"] \ar[dr,"1"', ""{name=V, below, sloped}] \ar[from=U, "\rho", Rightarrow] & \\
      \cat{C} & & \cat{C}
    \end{tikzcd}
  \end{equation*}
  the components of $\rho$ have type $\rho_a:a\boxtimes J\morph{} a$ as expected.
  A similar diagram is induced by $\lambda$ and in turn one sees that this has components $\lambda_a:J\boxtimes a\morph{} a$.
  Checking the triangle coherence equation follows like Theorem \ref{prop:monoidal_promonoidal}.

  Now suppose we begin with a promonoidal category $\cat{C}$ where the unit is representable $J(-)\cong \cat{C}(-,I)$ and for each $(b,c)\in\cat{C}\times\cat{C}$, either $\otimes(-,b,c)\cong \cat{C}(-,x_{bc})$ is representable, or $\otimes(-,b,c)\cong\varnothing(-)$ is empty.
  Define a full subcategory $\cat{S}$ of $\cat{C}\times\cat{C}$ spanned by objects $(b,c)$ where $\otimes(-,b,c)$ is representable.
  Suppose for a contradiction that $(b,c)\in\cat{S}$ and there exists a $(f,g):(b,c)\morph{}(b',c')$ in $\cat{C}\times\cat{C}$ but with $(b',c')\notin\cat{S}$.
  Then we would have a natural transformation $\cat{C}(-,x_{bc})\morph{}\varnothing(-)$, a contradiction.
  Thus $(f,g)$ cannot exist and as a result the canonical inclusion functor $\cat{S}\hookrightarrow\cat{C}\times\cat{C}$ is a discrete opfibration.
\end{proof}

\end{document}

%% file: preamble.tex
\usepackage[utf8]{inputenc}
\usepackage{amsmath}
\usepackage{amsthm}
\usepackage{amssymb}
\usepackage{hyperref}
\usepackage{physics}
\usepackage{stmaryrd}
\usepackage{tikz}
\usepackage{tikz-cd}
\usepackage{tikzit}
\usepackage{circuitikz}
\usepackage[numbers]{natbib}
\usepackage[T1]{fontenc}
\usepackage{url}
\usepackage{caption}
\usepackage{subcaption}

\input{styles.tikzdefs}
\input{styles.tikzstyles}
\usetikzlibrary{circuits.ee.IEC}

\newcommand{\morph}[1]{\xrightarrow{#1}}
\newcommand{\lmorph}[1]{\xleftarrow{#1}}
\newcommand{\pmorph}{\relbar\joinrel\mapstochar\joinrel\rightarrow}

\newcommand{\cat}[1]{\mathcal{#1}}

\newcommand{\opcat}[1]{{#1}^\textrm{op}}

\newcommand{\set}{\mathsf{Set}}

\newcommand{\reals}{\mathbb{R}}
\newcommand{\Space}{\mathsf{Space}}
\newcommand{\Slice}{\mathsf{Slice}}
\newcommand{\Cat}{\mathsf{Cat}}
\newcommand{\Prof}{\mathsf{Prof}}
\newcommand{\pcat}{\mathsf{PCat}}

\newcommand{\funny}{\mathrel{\square}}
\newcommand{\oor}{\varovee}
\newcommand{\oand}{\varowedge}

\DeclareFontFamily{U}{min}{}
\DeclareFontShape{U}{min}{m}{n}{<-> udmj30}{}
\newcommand\yo{\!\text{\usefont{U}{min}{m}{n}\symbol{'210}}\!}

\theoremstyle{definition}
\newtheorem{defn}{Definition}
\theoremstyle{plain}
\newtheorem{prop}{Proposition}
\theoremstyle{plain}
\newtheorem{lem}{Lemma}
\theoremstyle{plain}
\newtheorem{thm}{Theorem}
\theoremstyle{plain}

\theoremstyle{remark}
\newtheorem*{remark}{Remark}
\theoremstyle{definition}

%% file: styles.tikzstyles
\tikzstyle{discarding}=[fill=white, draw=black, shape=circle, style=upground]
\tikzstyle{smalldiscarding}=[fill=white, draw=black, style=upground, scale=0.6]
\tikzstyle{backdiscard}=[fill=white, draw=black, shape=circle, style=downground, scale=0.5]
\tikzstyle{smallbackdiscard}=[fill=white, draw=black, shape=circle, style=downground, scale=0.5]
\tikzstyle{state}=[fill=white, draw=black, style=triang, tikzit shape=rectangle]
\tikzstyle{state_small}=[fill=white, draw=black, style={triang_small}, tikzit shape=rectangle]
\tikzstyle{state_lesspad}=[fill=white, draw=black, style={triang_lesssep}, tikzit shape=rectangle]
\tikzstyle{kstate}=[fill=white, draw=black, style=kpoint, tikzit shape=rectangle]
\tikzstyle{kstateconj}=[fill=white, draw=black, style=kpoint conjugate, tikzit shape=rectangle]
\tikzstyle{kstateBIG}=[fill=white, draw=black, style=big kpoint, tikzit shape=rectangle]
\tikzstyle{effect}=[fill=white, draw=black, style=triangdag]
\tikzstyle{effect_small}=[fill=white, draw=black, style={triangdag_small}]
\tikzstyle{keffect}=[fill=white, draw=black, style=kpoint adjoint]
\tikzstyle{keffectconj}=[fill=white, draw=black, style=kpoint transpose]
\tikzstyle{morphdag}=[style=mapdag]
\tikzstyle{morph}=[style=map]
\tikzstyle{morphtrans}=[style=maptrans]
\tikzstyle{morphconj}=[style=mapconj]
\tikzstyle{CPMmorph}=[style=dmap]
\tikzstyle{CPMmorphconj}=[style=dmapconj]
\tikzstyle{CPMmorphdag}=[style=dmapdag]
\tikzstyle{CPMmorphtrans}=[style=dmaptrans]
\tikzstyle{CPMstate}=[fill=white, draw=black, style=triang, doubled]
\tikzstyle{CPMstateBIG}=[fill=white, draw=black, style={triang_lesssep}, doubled]
\tikzstyle{CPMkstate}=[fill=white, draw=black, style=kpoint, tikzit shape=rectangle, doubled]
\tikzstyle{CPMkstateconj}=[fill=white, draw=black, style=kpoint conjugate, tikzit shape=rectangle, doubled]
\tikzstyle{CPMkstateBIG}=[fill=white, draw=black, style=big kpoint, tikzit shape=rectangle, doubled]
\tikzstyle{CPMkeffect}=[fill=white, draw=black, style=kpoint adjoint, doubled]
\tikzstyle{CPMkeffectconj}=[fill=white, draw=black, style=kpoint transpose, doubled]
\tikzstyle{UHfB}=[fill=white, draw=black, style=triangdag, doubled, inner sep=-2pt]
\tikzstyle{leak}=[style=tinypoint, regular polygon rotate=-90]
\tikzstyle{leakfill}=[style=tinypoint, regular polygon rotate=-90, fill=black]
\tikzstyle{Z}=[style=dot, fill=green]
\tikzstyle{X}=[style=dot, fill=red]
\tikzstyle{black_dot}=[style=dot, fill=black]
\tikzstyle{white_dot}=[style=dot, fill=white]
\tikzstyle{qblack_dot}=[style=ddot, fill=black]
\tikzstyle{qwhite_dot}=[style=ddot, fill=white]
\tikzstyle{whitephase}=[style=wphase dot, fill=white]
\tikzstyle{qredphase}=[style=phase dot, fill=red]
\tikzstyle{qgreenphase}=[style=phase dot, fill=green]
\tikzstyle{CPMbox}=[style=hadamard, doubled]
\tikzstyle{box}=[style=hadamard]
\tikzstyle{bigbox}=[style=hadamard, minimum height=6mm, minimum width=8mm]
\tikzstyle{widebigbox}=[style=hadamard, minimum height=6mm, minimum width=18mm]
\tikzstyle{WIDEBOI}=[style=hadamard, minimum height=6mm, minimum width=40mm]
\tikzstyle{WIDEstate}=[fill=white, draw=black, style=triang, tikzit shape=rectangle, minimum width=40mm]
\tikzstyle{bigboxCPM}=[style=hadamard, minimum height=6mm, minimum width=8mm, doubled]
\tikzstyle{antipode}=[style=anti]
\tikzstyle{WIDEcutstate}=[style={state_doublecut}, text depth=2mm, minimum width=8mm]
\tikzstyle{WIDERcutstate}=[style={state_doublecut}, text depth=0mm, minimum width=36mm]
\tikzstyle{WIDEcutstateCPM}=[style={state_doublecut}, text depth=2mm, minimum width=8mm, doubled]
\tikzstyle{WIDEcuteffect}=[style={effect_doublecut}, text depth=2mm, minimum width=8mm]

\tikzstyle{dottededge}=[-, dotted]
\tikzstyle{double edge}=[-, style=doubled, draw=black, tikzit draw={rgb,255: red,191; green,0; blue,64}]
\tikzstyle{red_shade}=[-, draw=none, tikzit draw={rgb,255: red,132; green,175; blue,175}, fill={rgb,255: red,255; green,67; blue,95}, opacity=0.5]
\tikzstyle{green_shade}=[-, draw=none, tikzit draw={rgb,255: red,132; green,175; blue,175}, fill={rgb,255: red,111; green,255; blue,145}, opacity=0.5]
\tikzstyle{green_shade_border}=[-, tikzit draw={rgb,255: red,132; green,175; blue,175}, fill={rgb,255: red,111; green,255; blue,145}, opacity=0.5]
\tikzstyle{light_green_shade_border}=[-, draw=none,tikzit draw={rgb,255: red,132; green,175; blue,175}, fill={rgb,255: red,111; green,215; blue,145}, opacity=0.5]
\tikzstyle{blue_shade}=[-, draw=none, tikzit draw={rgb,255: red,132; green,175; blue,175}, fill={rgb,255: red,106; green,171; blue,255}, opacity=0.5]
\tikzstyle{orange_shade}=[-, draw=none, tikzit draw={rgb,255: red,132; green,175; blue,175}, fill={rgb,255: red,255; green,158; blue,78}]
\tikzstyle{purple_shade}=[-, draw=none, tikzit draw={rgb,255: red,132; green,175; blue,175}, fill={rgb,255: red,203; green,71; blue,255}, opacity=0.5]
\tikzstyle{yellow_shade}=[-, draw=none, tikzit draw={rgb,255: red,132; green,175; blue,175}, fill={rgb,255: red,255; green,250; blue,101}]
\tikzstyle{greenpath}=[-, draw={rgb,255: red,111; green,255; blue,145}]
\tikzstyle{blue path}=[-, draw={rgb,255: red,106; green,171; blue,255}]
\tikzstyle{brace edge}=[-, decorate, decoration={{brace,amplitude=1mm,raise=-1mm}}]

%% file: figures/premonoidal.tikz
\begin{tikzpicture}
	\begin{pgfonlayer}{nodelayer}
		\node [style=box] (0) at (-0.5, -0.5) {$f$};
		\node [style=box] (1) at (0.5, 0.5) {$g$};
		\node [style=box] (2) at (2.5, 0.5) {$f$};
		\node [style=box] (3) at (3.5, -0.5) {$g$};
		\node [style=none] (4) at (-0.5, 1.5) {};
		\node [style=none] (5) at (0.5, 1.5) {};
		\node [style=none] (6) at (-0.5, -1.5) {};
		\node [style=none] (7) at (0.5, -1.5) {};
		\node [style=none] (8) at (2.5, 1.5) {};
		\node [style=none] (9) at (3.5, 1.5) {};
		\node [style=none] (10) at (2.5, -1.5) {};
		\node [style=none] (11) at (3.5, -1.5) {};
		\node [style=none] (12) at (1.5, 0) {$\neq$};
	\end{pgfonlayer}
	\begin{pgfonlayer}{edgelayer}
		\draw (4.center) to (0);
		\draw (0) to (6.center);
		\draw (5.center) to (1);
		\draw (1) to (7.center);
		\draw (8.center) to (2);
		\draw (2) to (10.center);
		\draw (9.center) to (3);
		\draw (3) to (11.center);
	\end{pgfonlayer}
\end{tikzpicture}

%% file: figures/morphism1.tikz
\begin{tikzpicture}
	\begin{pgfonlayer}{nodelayer}
		\node [style=none] (124) at (-2.25, 3.75) {};
		\node [style=none] (125) at (-3.5, -3.75) {};
		\node [style=none] (126) at (-4, -3.75) {};
		\node [style=none] (127) at (-3.75, 3.75) {};
		\node [style=none] (11) at (5.25, -0.75) {};
		\node [style=none] (12) at (2.75, -3.75) {};
		\node [style=none] (13) at (6.25, -3.75) {};
		\node [style=none] (4) at (4.5, 0.75) {};
		\node [style=none] (5) at (5.25, -0.75) {};
		\node [style=none] (6) at (6.75, 3.75) {};
		\node [style=none] (7) at (3.25, 3.75) {};
		\node [style=none] (0) at (3.5, -0.75) {};
		\node [style=none] (1) at (5.25, -0.75) {};
		\node [style=none] (2) at (6.75, 3.75) {};
		\node [style=none] (3) at (1.25, 3.75) {};
		\node [style=none] (8) at (2.25, -3.75) {};
		\node [style=none] (9) at (6.25, -3.75) {};
		\node [style=none] (14) at (0, 0) {$\subseteq$};
		\node [style=none] (91) at (4.5, 0.75) {};
		\node [style=none] (92) at (6.25, 0.75) {};
		\node [style=none] (93) at (7.5, 3.75) {};
		\node [style=none] (94) at (3.25, 3.75) {};
		\node [style=none] (95) at (2.75, -3.75) {};
		\node [style=none] (96) at (7.75, -3.75) {};
		\node [style=none] (97) at (6.25, -3.75) {};
		\node [style=none] (98) at (5.25, -0.75) {};
		\node [style=none] (99) at (6.75, 3.75) {};
		\node [style=none] (100) at (-3.5, -0.75) {};
		\node [style=none] (102) at (-2.5, -3.75) {};
		\node [style=none] (104) at (-3.5, -0.75) {};
		\node [style=none] (105) at (-2, 3.75) {};
		\node [style=none] (107) at (-5.25, -0.75) {};
		\node [style=none] (108) at (-3.5, -0.75) {};
		\node [style=none] (109) at (-2, 3.75) {};
		\node [style=none] (110) at (-7.5, 3.75) {};
		\node [style=none] (111) at (-6.5, -3.75) {};
		\node [style=none] (112) at (-2.5, -3.75) {};
		\node [style=none] (114) at (-4.25, 0.75) {};
		\node [style=none] (115) at (-2.5, 0.75) {};
		\node [style=none] (116) at (-1.25, 3.75) {};
		\node [style=none] (117) at (-5.5, 3.75) {};
		\node [style=none] (118) at (-6, -3.75) {};
		\node [style=none] (119) at (-1, -3.75) {};
		\node [style=none] (120) at (-5.75, -0.75) {$X$};
		\node [style=none] (121) at (-4.75, 0.75) {$Y$};
		\node [style=none] (122) at (3, -0.75) {$X$};
		\node [style=none] (123) at (4, 0.75) {$Y$};
	\end{pgfonlayer}
	\begin{pgfonlayer}{edgelayer}
		\draw [style={blue_shade}] (124.center)
			 to [in=105, out=-105] (125.center)
			 to (126.center)
			 to [in=-75, out=105] (127.center)
			 to cycle;
		\draw [style={blue_shade}] (11.center) to (13.center);
		\draw [style={blue_shade}] (7.center)
			 to (4.center)
			 to (12.center)
			 to (97.center)
			 to (98.center)
			 to (99.center);
		\draw [style={blue_shade}] (5.center) to (6.center);
		\draw [style=double edge, in=150, out=-15] (0.center) to (1.center);
		\draw [style=dottededge] (1.center) to (2.center);
		\draw [style=dottededge] (0.center) to (3.center);
		\draw [style=dottededge] (9.center) to (1.center);
		\draw [style=dottededge] (8.center) to (0.center);
		\draw [style=double edge, in=150, out=-15] (91.center) to (92.center);
		\draw [style=dottededge] (92.center) to (93.center);
		\draw [style=dottededge] (91.center) to (94.center);
		\draw [style=dottededge] (96.center) to (92.center);
		\draw [style=dottededge] (95.center) to (91.center);
		\draw [style={blue_shade}] (100.center) to (102.center);
		\draw [style={blue_shade}] (104.center) to (105.center);
		\draw [style=double edge, in=150, out=-15] (107.center) to (108.center);
		\draw [style=dottededge] (108.center) to (109.center);
		\draw [style=dottededge] (107.center) to (110.center);
		\draw [style=dottededge] (112.center) to (108.center);
		\draw [style=dottededge] (111.center) to (107.center);
		\draw [style=double edge, in=150, out=-15] (114.center) to (115.center);
		\draw [style=dottededge] (115.center) to (116.center);
		\draw [style=dottededge] (114.center) to (117.center);
		\draw [style=dottededge] (119.center) to (115.center);
		\draw [style=dottededge] (118.center) to (114.center);
	\end{pgfonlayer}
\end{tikzpicture}

%% file: figures/compose.tikz
\begin{tikzpicture}
	\begin{pgfonlayer}{nodelayer}
		\node [style=none] (0) at (-3, -1.25) {};
		\node [style=none] (1) at (-0.75, -1.25) {};
		\node [style=none] (2) at (1.05, 4.5) {};
		\node [style=none] (3) at (-5, 4.5) {};
		\node [style=none] (8) at (-4, -4.25) {};
		\node [style=none] (9) at (0.25, -4.25) {};
		\node [style=none] (14) at (-2.25, 0) {};
		\node [style=none] (15) at (0, 0) {};
		\node [style=none] (16) at (1.5, -4.25) {};
		\node [style=none] (17) at (-3.5, -4.25) {};
		\node [style=none] (18) at (-3.75, 4.5) {};
		\node [style=none] (19) at (1.5, 4.5) {};
		\node [style=none] (20) at (-1, 1.5) {};
		\node [style=none] (21) at (-3.25, 1.5) {};
		\node [style=none] (22) at (-5, -4.25) {};
		\node [style=none] (23) at (1.05, -4.25) {};
		\node [style=none] (24) at (0, 4.5) {};
		\node [style=none] (25) at (-4.25, 4.5) {};
		\node [style=none] (26) at (-1.25, 4.5) {};
		\node [style=none] (27) at (-3.25, 4.5) {};
		\node [style=none] (28) at (-0.75, -0.125) {};
		\node [style=none] (29) at (-2, 0.25) {};
		\node [style=none] (30) at (-1.75, -4.25) {};
		\node [style=none] (31) at (0.75, -4.25) {};
		\node [style=none] (32) at (-2, 4.5) {};
		\node [style=none] (33) at (0.5, 4.5) {};
		\node [style=none] (34) at (-2, 0.25) {};
		\node [style=none] (35) at (-0.75, -0.125) {};
		\node [style=none] (36) at (-1, -4.25) {};
		\node [style=none] (37) at (-3, -4.25) {};
		\node [style=none] (38) at (2.5, 0) {$\mapsto$};
		\node [style=none] (39) at (5.5, -1.25) {};
		\node [style=none] (40) at (7.75, -1.25) {};
		\node [style=none] (41) at (9.55, 4.5) {};
		\node [style=none] (42) at (3.5, 4.5) {};
		\node [style=none] (43) at (4.5, -4.25) {};
		\node [style=none] (44) at (8.75, -4.25) {};
		\node [style=none] (47) at (10, -4.25) {};
		\node [style=none] (48) at (5, -4.25) {};
		\node [style=none] (49) at (4.75, 4.5) {};
		\node [style=none] (50) at (10, 4.5) {};
		\node [style=none] (51) at (7.5, 1.5) {};
		\node [style=none] (52) at (5.25, 1.5) {};
		\node [style=none] (53) at (3.5, -4.25) {};
		\node [style=none] (54) at (9.55, -4.25) {};
		\node [style=none] (55) at (8.5, 4.5) {};
		\node [style=none] (56) at (4.25, 4.5) {};
		\node [style=none] (57) at (7.25, 4.5) {};
		\node [style=none] (61) at (6.75, -4.25) {};
		\node [style=none] (63) at (6.5, 4.5) {};
		\node [style=none] (65) at (6.5, 0.25) {};
		\node [style=none] (66) at (7.675, -0.125) {};
		\node [style=none] (67) at (7.5, -4.25) {};
		\node [style=none] (68) at (-2, -4.9) {\color{blue}\footnotesize$S : X \to Y$};
		\node [style=none] (69) at (-2.25, 5.1) {\color{red}\footnotesize$T : Y \to Z$};
		\node [style=none] (70) at (-0.75, -4.5) {};
		\node [style=none] (71) at (-3.25, -4.5) {};
		\node [style=none] (72) at (-3.5, 4.75) {};
		\node [style=none] (73) at (-1, 4.75) {};
		\node [style=none] (74) at (7.75, -4.9) {\color{magenta}\footnotesize$T \circ S : X \to Z$};
		\node [style=none] (75) at (7.75, -4.5) {};
		\node [style=none] (76) at (6.5, -4.5) {};
		\node [style=none] (77) at (-3.75, 1.5) {$Z$};
		\node [style=none] (78) at (-2.75, 0) {$Y$};
		\node [style=none] (79) at (-3.5, -1.25) {$X$};
		\node [style=none] (80) at (5, -1.25) {$X$};
		\node [style=none] (81) at (4.75, 1.5) {$Z$};
	\end{pgfonlayer}
	\begin{pgfonlayer}{edgelayer}
		\draw [style=dottededge] (0.center) to (3.center);
		\draw [style=dottededge] (8.center) to (0.center);
		\draw [style=dottededge] (15.center) to (16.center);
		\draw [style=dottededge] (14.center) to (17.center);
		\draw [style=dottededge] (19.center) to (15.center);
		\draw [style=dottededge] (18.center) to (14.center);
		\draw [style=dottededge] (21.center) to (22.center);
		\draw [style=dottededge] (20.center) to (23.center);
		\draw [style=dottededge] (25.center) to (21.center);
		\draw [style=dottededge] (24.center) to (20.center);
		\draw [style={red_shade}] (31.center)
			 to [in=-75, out=105] (28.center)
			 to [in=-105, out=105] (26.center)
			 to (27.center)
			 to [in=90, out=-75] (29.center)
			 to [in=90, out=-90] (30.center);
		\draw [style=double edge, in=-150, out=15] (14.center) to (15.center);
		\draw [style=double edge, in=-30, out=165] (20.center) to (21.center);
		\draw [style=double edge, in=150, out=-15] (0.center) to (1.center);
		\draw [style=dottededge] (1.center) to (2.center);
		\draw [style=dottededge] (9.center) to (1.center);
		\draw [style={blue_shade}] (37.center)
			 to [in=-90, out=75] (34.center)
			 to [in=-105, out=90] (32.center)
			 to (33.center)
			 to [in=90, out=-105] (35.center)
			 to [in=90, out=-90] (36.center);
		\draw [style=dottededge] (39.center) to (42.center);
		\draw [style=dottededge] (43.center) to (39.center);
		\draw [style=dottededge] (52.center) to (53.center);
		\draw [style=dottededge] (51.center) to (54.center);
		\draw [style=dottededge] (56.center) to (52.center);
		\draw [style=dottededge] (55.center) to (51.center);
		\draw [style={purple_shade}] (63.center)
			 to [in=90, out=-105] (65.center)
			 to [in=90, out=-90] (61.center)
			 to (67.center)
			 to [in=-90, out=90] (66.center)
			 to [in=-105, out=105, looseness=0.75] (57.center);
		\draw [style=double edge, in=-30, out=165] (51.center) to (52.center);
		\draw [style=double edge, in=150, out=-15] (39.center) to (40.center);
		\draw [style=dottededge] (40.center) to (41.center);
		\draw [style=dottededge] (44.center) to (40.center);
		\draw [style=brace edge, draw=blue] (70.center) to (71.center);
		\draw [style=brace edge, draw=red] (72.center) to (73.center);
		\draw [style=brace edge, draw=magenta] (75.center) to (76.center);
	\end{pgfonlayer}
\end{tikzpicture}

%% file: figures/conjunction.tikz
\begin{tikzpicture}
	\begin{pgfonlayer}{nodelayer}
		\node [style=none] (32) at (-2.25, 4.5) {};
		\node [style=none] (33) at (-0.5, 4.5) {};
		\node [style=none] (34) at (-2, 0.25) {};
		\node [style=none] (35) at (-1.25, -0.125) {};
		\node [style=none] (36) at (-1, -4.25) {};
		\node [style=none] (37) at (-3, -4.25) {};
		\node [style=none] (117) at (8.075, -0.15) {};
		\node [style=none] (109) at (7.75, -1.25) {};
		\node [style=none] (111) at (8.75, -4.25) {};
		\node [style=none] (112) at (6.25, 0) {};
		\node [style=none] (113) at (4.75, 4.5) {};
		\node [style=none] (114) at (7.5, 1.5) {};
		\node [style=none] (115) at (8.5, 4.5) {};
		\node [style=none] (116) at (5, -4.25) {};
		\node [style=none] (0) at (-3, -1.25) {};
		\node [style=none] (1) at (-0.75, -1.25) {};
		\node [style=none] (2) at (1.05, 4.5) {};
		\node [style=none] (3) at (-5, 4.5) {};
		\node [style=none] (8) at (-4, -4.25) {};
		\node [style=none] (9) at (0.25, -4.25) {};
		\node [style=none] (14) at (-2.25, 0) {};
		\node [style=none] (15) at (0, 0) {};
		\node [style=none] (16) at (1.5, -4.25) {};
		\node [style=none] (17) at (-3.5, -4.25) {};
		\node [style=none] (18) at (-3.75, 4.5) {};
		\node [style=none] (19) at (1.5, 4.5) {};
		\node [style=none] (20) at (-1, 1.5) {};
		\node [style=none] (21) at (-3.25, 1.5) {};
		\node [style=none] (22) at (-5, -4.25) {};
		\node [style=none] (23) at (1.05, -4.25) {};
		\node [style=none] (24) at (0, 4.5) {};
		\node [style=none] (25) at (-4.25, 4.5) {};
		\node [style=none] (29) at (-2, 0.25) {};
		\node [style=none] (38) at (2.5, 0) {$\subseteq$};
		\node [style=none] (77) at (-3.75, 1.5) {$Y$};
		\node [style=none] (78) at (-2.75, 0) {$X$};
		\node [style=none] (79) at (-3.5, -1.25) {$Z$};
		\node [style=none] (80) at (5.5, -1.25) {};
		\node [style=none] (81) at (7.75, -1.25) {};
		\node [style=none] (82) at (9.55, 4.5) {};
		\node [style=none] (83) at (3.5, 4.5) {};
		\node [style=none] (84) at (4.5, -4.25) {};
		\node [style=none] (85) at (8.75, -4.25) {};
		\node [style=none] (86) at (6.25, 0) {};
		\node [style=none] (87) at (8.5, 0) {};
		\node [style=none] (88) at (10, -4.25) {};
		\node [style=none] (89) at (5, -4.25) {};
		\node [style=none] (90) at (4.75, 4.5) {};
		\node [style=none] (91) at (10, 4.5) {};
		\node [style=none] (92) at (7.5, 1.5) {};
		\node [style=none] (93) at (5.25, 1.5) {};
		\node [style=none] (94) at (3.5, -4.25) {};
		\node [style=none] (95) at (9.55, -4.25) {};
		\node [style=none] (96) at (8.5, 4.5) {};
		\node [style=none] (97) at (4.25, 4.5) {};
		\node [style=none] (98) at (6.5, 0.25) {};
		\node [style=none] (105) at (4.75, 1.5) {$Y$};
		\node [style=none] (106) at (5.75, 0) {$X$};
		\node [style=none] (107) at (5, -1.25) {$Z$};
	\end{pgfonlayer}
	\begin{pgfonlayer}{edgelayer}
		\draw [style={blue_shade}] (37.center)
			 to [in=-90, out=75] (34.center)
			 to [in=-105, out=90] (32.center)
			 to (33.center)
			 to [in=90, out=-105] (35.center)
			 to [in=90, out=-90] (36.center);
		\draw [style={blue_shade}] (113.center)
			 to (112.center)
			 to (116.center)
			 to (111.center)
			 to (109.center)
			 to (117.center)
			 to (114.center)
			 to (115.center);
		\draw [style=dottededge] (0.center) to (3.center);
		\draw [style=dottededge] (8.center) to (0.center);
		\draw [style=dottededge] (15.center) to (16.center);
		\draw [style=dottededge] (14.center) to (17.center);
		\draw [style=dottededge] (19.center) to (15.center);
		\draw [style=dottededge] (18.center) to (14.center);
		\draw [style=dottededge] (21.center) to (22.center);
		\draw [style=dottededge] (20.center) to (23.center);
		\draw [style=dottededge] (25.center) to (21.center);
		\draw [style=dottededge] (24.center) to (20.center);
		\draw [style=double edge, in=-150, out=15] (14.center) to (15.center);
		\draw [style=double edge, in=-30, out=165] (20.center) to (21.center);
		\draw [style=double edge, in=150, out=-15] (0.center) to (1.center);
		\draw [style=dottededge] (1.center) to (2.center);
		\draw [style=dottededge] (9.center) to (1.center);
		\draw [style=dottededge] (80.center) to (83.center);
		\draw [style=dottededge] (84.center) to (80.center);
		\draw [style=dottededge] (87.center) to (88.center);
		\draw [style=dottededge] (86.center) to (89.center);
		\draw [style=dottededge] (91.center) to (87.center);
		\draw [style=dottededge] (90.center) to (86.center);
		\draw [style=dottededge] (93.center) to (94.center);
		\draw [style=dottededge] (92.center) to (95.center);
		\draw [style=dottededge] (97.center) to (93.center);
		\draw [style=dottededge] (96.center) to (92.center);
		\draw [style=double edge, in=-150, out=15] (86.center) to (87.center);
		\draw [style=double edge, in=-30, out=165] (92.center) to (93.center);
		\draw [style=double edge, in=150, out=-15] (80.center) to (81.center);
		\draw [style=dottededge] (81.center) to (82.center);
		\draw [style=dottededge] (85.center) to (81.center);
	\end{pgfonlayer}
\end{tikzpicture}

%% file: figures/disjunction.tikz
\begin{tikzpicture}
	\begin{pgfonlayer}{nodelayer}
		\node [style=none] (147) at (0.25, 4.5) {};
		\node [style=none] (148) at (-1.5, -4.25) {};
		\node [style=none] (149) at (-2.25, -1.75) {};
		\node [style=none] (150) at (-2.5, 4.5) {};
		\node [style=none] (151) at (-3.75, 4.5) {};
		\node [style=none] (152) at (-4.5, 4.5) {};
		\node [style=none] (153) at (-2.75, -4.25) {};
		\node [style=none] (117) at (4.925, 0.425) {};
		\node [style=none] (109) at (5.5, -1.25) {};
		\node [style=none] (111) at (4.5, -4.25) {};
		\node [style=none] (112) at (7.75, -1.25) {};
		\node [style=none] (113) at (9.55, 4.5) {};
		\node [style=none] (114) at (5.25, 1.5) {};
		\node [style=none] (115) at (4.25, 4.5) {};
		\node [style=none] (116) at (8.75, -4.25) {};
		\node [style=none] (38) at (2.5, 0) {$\subseteq$};
		\node [style=none] (80) at (5.5, -1.25) {};
		\node [style=none] (81) at (7.75, -1.25) {};
		\node [style=none] (82) at (9.55, 4.5) {};
		\node [style=none] (83) at (3.5, 4.5) {};
		\node [style=none] (84) at (4.5, -4.25) {};
		\node [style=none] (85) at (8.75, -4.25) {};
		\node [style=none] (86) at (7, 0) {};
		\node [style=none] (87) at (9.25, 0) {};
		\node [style=none] (88) at (10.75, -4.25) {};
		\node [style=none] (89) at (5.75, -4.25) {};
		\node [style=none] (90) at (5.5, 4.5) {};
		\node [style=none] (91) at (10.75, 4.5) {};
		\node [style=none] (92) at (7.5, 1.5) {};
		\node [style=none] (93) at (5.25, 1.5) {};
		\node [style=none] (94) at (3.5, -4.25) {};
		\node [style=none] (95) at (9.55, -4.25) {};
		\node [style=none] (96) at (8.5, 4.5) {};
		\node [style=none] (97) at (4.25, 4.5) {};
		\node [style=none] (105) at (4.75, 1.5) {$Y$};
		\node [style=none] (106) at (9.75, 0) {$X$};
		\node [style=none] (107) at (5, -1.25) {$Z$};
		\node [style=none] (126) at (-3.5, -1.25) {};
		\node [style=none] (127) at (-1.25, -1.25) {};
		\node [style=none] (128) at (0.55, 4.5) {};
		\node [style=none] (129) at (-5.5, 4.5) {};
		\node [style=none] (130) at (-4.5, -4.25) {};
		\node [style=none] (131) at (-0.25, -4.25) {};
		\node [style=none] (132) at (-2, 0) {};
		\node [style=none] (133) at (0.25, 0) {};
		\node [style=none] (134) at (1.75, -4.25) {};
		\node [style=none] (135) at (-3.25, -4.25) {};
		\node [style=none] (136) at (-3.5, 4.5) {};
		\node [style=none] (137) at (1.75, 4.5) {};
		\node [style=none] (138) at (-1.5, 1.5) {};
		\node [style=none] (139) at (-3.75, 1.5) {};
		\node [style=none] (140) at (-5.5, -4.25) {};
		\node [style=none] (141) at (0.55, -4.25) {};
		\node [style=none] (142) at (-0.5, 4.5) {};
		\node [style=none] (143) at (-4.75, 4.5) {};
		\node [style=none] (144) at (-4.25, 1.5) {$Y$};
		\node [style=none] (145) at (0.75, 0) {$X$};
		\node [style=none] (146) at (-4, -1.25) {$Z$};
	\end{pgfonlayer}
	\begin{pgfonlayer}{edgelayer}
		\draw [style={blue_shade}] (153.center)
			 to [in=-75, out=75] (152.center)
			 to (151.center)
			 to [in=90, out=-90] (149.center)
			 to [in=-75, out=75] (150.center)
			 to (147.center)
			 to [in=105, out=-105] (148.center);
		\draw [style={blue_shade}] (113.center)
			 to (112.center)
			 to (116.center)
			 to (111.center)
			 to (109.center)
			 to (117.center)
			 to (114.center)
			 to (115.center);
		\draw [style=dottededge] (80.center) to (83.center);
		\draw [style=dottededge] (84.center) to (80.center);
		\draw [style=dottededge] (87.center) to (88.center);
		\draw [style=dottededge] (86.center) to (89.center);
		\draw [style=dottededge] (91.center) to (87.center);
		\draw [style=dottededge] (90.center) to (86.center);
		\draw [style=dottededge] (93.center) to (94.center);
		\draw [style=dottededge] (92.center) to (95.center);
		\draw [style=dottededge] (97.center) to (93.center);
		\draw [style=dottededge] (96.center) to (92.center);
		\draw [style=double edge, in=-150, out=15] (86.center) to (87.center);
		\draw [style=double edge, in=-30, out=165] (92.center) to (93.center);
		\draw [style=double edge, in=150, out=-15] (80.center) to (81.center);
		\draw [style=dottededge] (81.center) to (82.center);
		\draw [style=dottededge] (85.center) to (81.center);
		\draw [style=dottededge] (126.center) to (129.center);
		\draw [style=dottededge] (130.center) to (126.center);
		\draw [style=dottededge] (133.center) to (134.center);
		\draw [style=dottededge] (132.center) to (135.center);
		\draw [style=dottededge] (137.center) to (133.center);
		\draw [style=dottededge] (136.center) to (132.center);
		\draw [style=dottededge] (139.center) to (140.center);
		\draw [style=dottededge] (138.center) to (141.center);
		\draw [style=dottededge] (143.center) to (139.center);
		\draw [style=dottededge] (142.center) to (138.center);
		\draw [style=double edge, in=-150, out=15] (132.center) to (133.center);
		\draw [style=double edge, in=-30, out=165] (138.center) to (139.center);
		\draw [style=double edge, in=150, out=-15] (126.center) to (127.center);
		\draw [style=dottededge] (127.center) to (128.center);
		\draw [style=dottededge] (131.center) to (127.center);
	\end{pgfonlayer}
\end{tikzpicture}